\documentclass{amsart}
\usepackage{amssymb}
\usepackage[mathcal]{euscript}
\usepackage[cmtip,all]{xy}

\newcommand{\bydef}{:=}

\newcommand{\vphi}{\varphi}
\newcommand{\veps}{\varepsilon}

\newcommand{\wh}[1]{\widehat{#1}}
\newcommand{\wt}[1]{\widetilde{#1}}
\newcommand{\wb}[1]{\overline{#1}}

\newcommand{\sks}{\mathcal{K}}
\newcommand{\id}{\mathrm{id}}

\newcommand{\lspan}[1]{\mathrm{span}\left\{#1\right\}}
\newcommand{\matr}[1]{\left(\begin{smallmatrix}#1\end{smallmatrix}\right)}
\newcommand{\diag}{\mathrm{diag}}

\DeclareMathOperator*{\ot}{\otimes}

\newcommand{\op}{\mathrm{op}}
\DeclareMathOperator{\rank}{\mathrm{rank}} 

\newcommand{\bi}{\mathbf{i}}

\newcommand{\cA}{\mathcal{A}}

\newcommand{\cC}{\mathcal{C}}
\newcommand{\cD}{\mathcal{D}}

\newcommand{\cL}{\mathcal{L}}

\newcommand{\cO}{\mathcal{O}}

\newcommand{\cR}{\mathcal{R}}

\newcommand{\ZZ}{\mathbb{Z}}

\newcommand{\FF}{\mathbb{F}}

\newcommand{\chr}[1]{\mathrm{char}\,#1}

\DeclareMathOperator{\rad}{\mathrm{rad}}

\DeclareMathOperator{\End}{\mathrm{End}}

\DeclareMathOperator{\Aut}{\mathrm{Aut}}
\DeclareMathOperator{\inaut}{\mathrm{Int}}
\DeclareMathOperator{\Stab}{\mathrm{Stab}}

\newcommand{\brac}[1]{{#1}^{(-)}}
\newcommand{\ad}{\mathrm{ad}}
\newcommand{\Ad}{\mathrm{Ad}}

\newcommand{\Gl}{\mathfrak{gl}}
\newcommand{\Sl}{\mathfrak{sl}}

\newcommand{\So}{\mathfrak{so}}
\newcommand{\Sp}{\mathfrak{sp}}

\newcommand{\frso}{{\mathfrak{so}}}

\newcommand{\GL}{\mathrm{GL}}

\newcommand{\Ort}{\mathrm{O}}
\newcommand{\SO}{\mathrm{SO}}

\newcommand{\Spin}{\mathrm{Spin}}

\newcommand{\Br}{\mathrm{Br}}
\newcommand{\Ind}{\mathrm{Ind}}
\newcommand{\bG}{\wb{G}}
\newcommand{\bg}{\wb{g}}
\newcommand{\bT}{\wb{T}}
\newcommand{\bt}{\wb{t}}
\newcommand{\Cl}{\mathfrak{Cl}}
\newcommand{\GO}{\mathrm{GO}}
\newcommand{\PGO}{\mathrm{PGO}}
\newcommand{\subo}{_{\bar 0}}
\newcommand{\subuno}{_{\bar 1}}

\newtheorem{theorem}{Theorem}
\newtheorem{proposition}[theorem]{Proposition}
\newtheorem{lemma}[theorem]{Lemma}
\newtheorem{corollary}[theorem]{Corollary}

\theoremstyle{definition}
\newtheorem{df}[theorem]{Definition}
\newtheorem{example}[theorem]{Example}

\theoremstyle{remark}
\newtheorem{remark}[theorem]{Remark}

\begin{document}

\title{Graded modules over classical simple Lie algebras with a grading}

\author[A. Elduque]{Alberto Elduque${}^\star$}
\address{Departamento de Matem\'{a}ticas
 e Instituto Universitario de Matem\'aticas y Aplicaciones,
 Universidad de Zaragoza, 50009 Zaragoza, Spain}
\email{elduque@unizar.es}
\thanks{${}^\star$supported by the Spanish Ministerio de Econom\'{\i}a y Competitividad---Fondo Europeo de Desarrollo Regional (FEDER) MTM2010-18370-C04-02 and by the Diputaci\'on General de Arag\'on---Fondo Social Europeo (Grupo de Investigaci\'on de \'Algebra)}

\author[M. Kochetov]{Mikhail Kochetov${}^\dagger$}
\address{Department of Mathematics and Statistics,
 Memorial University of Newfoundland,
 St. John's, NL, A1C5S7, Canada}
\email{mikhail@mun.ca}
\thanks{${}^\dagger$supported by a sabbatical research grant of Memorial University and a grant for visiting scientists by Instituto Universitario de Matem\'aticas y Aplicaciones, University of Zaragoza}

\subjclass[2010]{Primary 17B70; Secondary 17B10, 16W50}

\keywords{Graded algebra, Lie algebra, graded module}

\date{}

\begin{abstract}
Given a grading by an abelian group $G$ on a semisimple Lie algebra $\cL$ over an algebraically closed field of characteristic $0$, we classify up to isomorphism the simple objects in the category of finite-dimensional $G$-graded $\cL$-modules. The invariants appearing in this classification are computed in the case when $\cL$ is simple classical (except for type $D_4$, where a partial result is given). In particular, we obtain criteria to determine when a finite-dimensional simple $\cL$-module admits a $G$-grading making it a graded $\cL$-module.
\end{abstract}

\maketitle

\section{Introduction}

Let $\cL$ be a semisimple finite-dimensional Lie algebra over an algebraically closed field $\FF$ of characteristic $0$. Suppose $\cL=\bigoplus_{g\in G}\cL_g$ is a grading by an abelian group $G$. We want to study finite-dimensional $G$-graded modules over $\cL$. Since $G$ is abelian, the universal  enveloping algebra $U(\cL)$ has a unique $G$-grading such that the canonical imbedding $\cL\to U(\cL)$ is a homomorphism of graded algebras. Thus, a graded $\cL$-module is the same as a graded $U(\cL)$-module. The following lemma is a well-known version of Maschke's Theorem in the graded setting:

\begin{lemma}
Let $\cA=\bigoplus_{g\in G}\cA_g$ be a $G$-graded associative algebra where $G$ is any group. Let $N\subset M$ be graded $\cA$-modules. If $N$ admits a complement in $M$ as an $\cA$-module then it admits a complement as a graded $\cA$-module.\qed
\end{lemma}

Applying this to a finite-dimensional graded module $W$ over $\cA=U(\cL)$, we conclude that $W$ is isomorphic to a direct sum of graded-simple $\cA$-modules. Thus we arrive to the problem of classifying graded-simple $\cL$-modules up to isomorphism. This turns out to be closely related to another natural problem: which of the $\cL$-modules admit a $G$-grading that makes them graded $\cL$-modules? 

The first step in solving these problems is a version of Clifford Theory (see Section \ref{s:Clifford_theory}), which describes graded-simple $\cL$-modules in terms of simple (ungraded) $\cL$-modules. In this context, there appears an action of the character group $\wh{G}$ on dominant integral weights of $\cL$ and what we call the graded Schur index of such a weight $\lambda$ (or of the corresponding simple $\cL$-module $V_\lambda$). In order to compute the Schur index, we use the graded Brauer group of $\FF$ (Section \ref{s:Brauer_group}). This is a special case of the so-called Brauer--Long group of a commutative ring \cite{Long74}, but since here the ring in question is $\FF$, an algebraically closed field, this group has a very simple structure, which allows efficient computation. To each dominant integral weight $\lambda$, we assign an element of the Brauer group, which will be called the Brauer invariant of $\lambda$ (or of $V_\lambda$) and denoted $\Br(\lambda)$. 

The second step, carried out for simple Lie algebras $\cL$ (Section \ref{s:Clifford_theory}), is a reduction of $\Br(\lambda)$ for general $\lambda$ to the case where $\lambda$ is the sum over a $\wh{G}$-orbit of fundamental weights. This orbit always has length $\le 3$, and length $3$ can appear only if $\cL$ has type $D_4$. Therefore, we concentrate on the case of one fundamental weight or a pair of fundamental weights positioned symmetrically on the Dynkin diagram of $\cL$.

Finally, we compute the Brauer invariants of all simple modules if $\cL$ is a simple Lie algebra of type $A_r$ ($r\ge 1$, Section \ref{s:A}), $B_r$ ($r\ge 2$, Section \ref{s:B}), $C_r$ ($r\ge 2$, Section \ref{s:C}) or $D_r$ ($r=3$ or $r>4$, Section \ref{s:D}), for all $G$-gradings on $\cL$. These $G$-gradings were classified up to isomorphism --- i.e., the action of $\Aut(\cL)$ --- in paper \cite{BK10} (see also the monograph \cite{EKmon} and references therein) in terms of the natural module for $\cL$, which is the simple module of minimal dimension (except for $B_2$ and $D_3$). As a by-product, we determine which gradings in series $D$ are ``inner'' and which are ``outer'', and classify them up to the action of $\inaut(\cL)$. (For series $A$, this was already clear in \cite{BK10}.) For type $D_4$, we restrict ourselves to ``matrix gradings'', i.e., those induced from the matrix algebra $M_8(\FF)$, hence an orbit of length $3$ cannot appear.

It turns out that the Brauer invariant of a simple module $V_\lambda$ is in most cases determined by the Brauer invariant of the natural module, which appears as a parameter in the classification of $G$-gradings on the Lie algebra. The exceptions are the following: type $A_r$ with an outer grading and odd $r\ge 3$ when $\lambda$ is symmetric and involves the central node of the Dynkin diagram, type $B_r$ when $\lambda$ involves the node corresponding to the spin module, and type $D_r$ when $\lambda$ involves the nodes corresponding to the half-spin modules. In these cases, we have to invoke other parameters in the classification of $G$-gradings. It is also worth noting that the order of $\Br(\lambda)$ in the graded Brauer group is almost always $\le 2$, the exceptions being type $A_r$ with an inner grading (then the order is a divisor of $r+1$) and type $D_r$ with an inner grading and odd $r$ (then the order is $1$, $2$ or $4$).

As to the exceptional simple Lie algebras, gradings have been classified up to isomorphism for types $G_2$ and $F_4$ (see \cite{EKmon} and references therein).  The situation is surprisingly simple in the case of  $G_2$, because the modules corresponding to its fundamental weights are the module of trace zero elements in the Cayley algebra and the adjoint module. Since any grading on $G_2$ is induced from a grading on the Cayley algebra, it follows that $\Br(\lambda)$ is trivial for any dominant integral weight $\lambda$, i.e., given any grading on $G_2$ by an abelian group, any module over $G_2$ admits a compatible grading. We do not consider the case of $F_4$ here.

All modules in this paper will be assumed {\em finite-dimensional} over $\FF$. Unless indicated otherwise, all vector spaces, algebras, tensor products, etc. will be taken over $\FF$.

\section{Graded Brauer group}\label{s:Brauer_group}

First we briefly recall gradings by abelian groups on matrix algebras --- see e.g. \cite{NVO82,E09d,BK10,EKmon} and references therein.

Let $G$ be a group and let $\cR=\bigoplus_{g\in G}\cR_g$ be a graded-simple finite-dimensional associative algebra over any field $\FF$. Then $\cR$ is isomorphic to $\End_\cD(W)$ where $\cD$ is a {\em graded division algebra} (i.e., every nonzero homogeneous element of $\cD$ is invertible) and $W$ is a finite-dimensional right ``vector space'' over $\cD$ (i.e., a graded right $\cD$-module, which is automatically free). The graded algebra $\cR$ has a unique graded-simple module, up to isomorphism and shift of grading. Hence, if $G$ is abelian, the isomorphism class of the graded algebra $\cD$ is determined by $\cR$. We will denote this class by $[\cR]$. Selecting a homogeneous $\cD$-basis $\{v_1,\ldots,v_k\}$ in $W$, $v_i\in W_{g_i}$, and setting $\wt{W}=\lspan{v_1,\ldots,v_k}$ (over $\FF$), we can write $\cR\cong\cC\ot\cD$ where $\cC=\End(\wt{W})$ is a matrix algebra with an {\em elementary grading}, i.e., a grading induced from its simple module. Explicitly, the elementary grading is determined by $(g_1,\ldots,g_k)$ as follows: $\deg E_{ij}=g_i g_j^{-1}$ where $E_{ij}$ are the matrix units associated to the basis $\{v_1,\ldots,v_k\}$.

Assume that $\FF$ is algebraically closed. Then $\cD$ is a twisted group algebra of its support $T$ (necessarily a subgroup of $G$), i.e., $\cD$ is spanned by elements $X_t$, $t\in T$, such that $X_s X_t=\sigma(s,t)X_{st}$ where $\sigma$ is a $2$-cocycle of $T$ with values in $\FF^\times$.  The elements $X_t$ are determined up to a scalar multiple and $\sigma$ up to a coboundary, but the relation
\[
X_s X_t=\beta(s,t)X_t X_s\quad\mbox{for all}\;s,t\in T
\]
implies that the function $\beta(s,t)=\sigma(s,t)/\sigma(t,s)$ is an alternating bicharacter of $G$ (with values in $\FF^\times$) and is uniquely determined by $\cD$. Moreover, $\cR$ (or $\cD$) is simple if and only if $\beta$ is {\em nondegenerate}, i.e., $\rad\beta:=\{t\in T\;|\;\beta(t,s)=1\;\forall s\in T\}$ is trivial. Note that if $\chr\FF=p$ then this condition forces $|T|$ to be coprime with $p$. Finally, if $G$ is abelian then the isomorphism class of the graded algebra $\cD$ is uniquely determined by the subgroup $T\subset G$ and the nondegenerate alternating bicharacter $\beta$, and, conversely, each such pair $(T,\beta)$  gives rise to a graded division algebra \cite{BK10}. Thus, for matrix algebras $\cR$ graded by an abelian group, the classes $[\cR]$ are in bijection with the pairs $(T,\beta)$.

Fix an abelian group $G$. Given $G$-graded matrix algebras $\cR_1=\End_{\cD_1}(W_1)$ and $\cR_2=\End_{\cD_2}(W_2)$ as above, the tensor product $\cR_1\ot\cR_2$ is again a $G$-graded matrix algebra and hence can be written in the form $\End_\cD(W)$. It is easy to see that $\cD$ is determined by $\cD_1$ and $\cD_2$. Indeed, writing $\cR_i=\cC_i\ot\cD_i$ ($i=1,2$) and $\cD_1\ot\cD_2=\cC\ot\cD$, we obtain $\cR_1\ot\cR_2=(\cC_1\ot \cC_2\ot \cC)\ot\cD$ where the first factor has an elementary grading and the second factor has a division grading. Hence we can unambiguously define $[\cR_1][\cR_2]$ as the isomorphism class of $\cD$. This gives an associative commutative multiplication on the set of classes, with the class of $\FF$ being the identity element and the class of $\cD^\op$ being the inverse of the class of $\cD$ (since $\cD\ot\cD^\op\cong\End(\cD)$ as graded algebras, and $\End(\cD)$ has an elementary grading induced by the grading of $\cD$). Thus we obtain an abelian group, which will be called the {\em $G$-graded Brauer group of $\FF$}.

In order to express $\cD$ in terms of $\cD_1$ and $\cD_2$, it will be convenient to rewrite the gradings in terms of actions. The group of characters $\wh{G}$ acts on any $G$-graded vector space $V$:
\begin{equation}\label{eq:hat_action}
\chi*v=\chi(g)v\quad\mbox{for all}\; v\in V_g, g\in G,\chi\in\wh{G}.
\end{equation}
If $\chr\FF=0$ or if $\chr\FF=p$ and $G$ has no $p$-torsion then the grading can be recovered as the eigenspace decomposition relative to this action. If $\cR$ is a matrix algebra with a $G$-grading then each $\chi\in\wh{G}$ acts by an automorphism of $\cR$, so there exists an invertible element $u_\chi\in\cR$ such that $\chi*x=u_\chi x u_\chi^{-1}$ for all $x\in\cR$ (Noether--Skolem Theorem). Explicitly, if $\cR=\cC\ot\cD$ as above then we can take
\[
u_\chi=\diag(\chi(g_1),\ldots,\chi(g_k))\ot X_t
\]
where $t$ is the unique element of $T$ such that $\chi(s)=\beta(t,s)$ for all $s\in T$. Note that $u_{\chi_1} u_{\chi_2}=\beta(t_1,t_2) u_{\chi_2} u_{\chi_1}$. Define an alternating bicharacter $\hat{\beta}$ on $\wh{G}$ (with values in $\FF^\times$) by setting $\hat{\beta}(\chi_1,\chi_2)=\beta(t_1,t_2)$. Thus we have
\begin{equation}\label{eq:characterization_hat_beta}
u_{\chi_1} u_{\chi_2}=\hat{\beta}(\chi_1,\chi_2) u_{\chi_2} u_{\chi_1}\quad\mbox{for all}\;\chi_1,\chi_2\in\wh{G}.
\end{equation}
This ``commutation factor'' $\hat{\beta}$ is in fact the obstruction preventing the mapping $\chi\mapsto u_\chi$ from being a representation of $\wh{G}$ (cf. \cite{Zol02}). For our purposes, it is important that $T$ and $\beta$ can be recovered from $\hat{\beta}$ as follows: $T=(\rad\hat{\beta})^\perp$ and $\beta(t_1,t_2)=\hat\beta(\chi_1,\chi_2)$ where $\chi_i$ is any character such that $\hat\beta(\psi,\chi_i)=\psi(t_i)$ for all $\psi\in\wh{G}$ ($i=1,2$).
It is clear from the characterization \eqref{eq:characterization_hat_beta} that the ``commutation factor'' of $\cR_1\ot\cR_2$ is the product of those of $\cR_1$ and $\cR_2$. Hence the class $[\cR_1\ot\cR_2]$ is given by $T$ and $\beta$ determined by $\hat{\beta}=\hat{\beta}_1\hat{\beta}_2$ as above.

We can summarize this discussion by saying that the $G$-graded Brauer group of $\FF$ is isomorphic to the group of alternating continuous bicharacters of the pro-finite group $\wh{G_0}$ where $G_0$ is the torsion subgroup of $G$ if $\chr\FF=0$ and the $p'$-torsion subgroup of $G$ if $\chr\FF=p$ (i.e., the set of all elements whose order is finite and coprime with $p$). The topology of $\wh{G_0}$, which makes it a compact and totally discontinuous topological group, comes from the identification of $\wh{G_0}$ with the inverse limit of the finite groups $\wh{H}$ where $H$ ranges over all finite subgroups of $G_0$. Equivalently, the topology of $\wh{G_0}$ is given by the system of neighborhoods of identity consisting of the subgroups of finite index $H^\perp$ for the same $H$. This topology allows us to retain in this setting all the usual properties of duality for finite abelian groups: we just have to restrict our attention to {\em continuous} characters $\wh{G_0}\to\FF^\times$ (where $\FF$ has discrete topology) and {\em closed} subgroups of $\wh{G_0}$. In particular, any continuous alternating bicharacter $\gamma\colon\wh{G_0}\times\wh{G_0}\to\FF^\times$  has the form $\hat{\beta}$ for some nondegenerate alternating bicharacter $\beta$ of the finite subgroup $(\rad\gamma)^\perp\subset G_0$.

The following property will be useful in the sequel:

\begin{lemma}\label{lm:corner}
Let $\cR$ be a matrix algebra over an algebraically closed field $\FF$. Suppose $\cR$ is graded by an abelian group $G$. If $\veps$ is a homogeneous idempotent of $\cR$ then $\veps\cR\veps$ is a $G$-graded matrix algebra and $[\veps\cR\veps]=[\cR]$ in the $G$-graded Brauer group of $\FF$.
\end{lemma}

\begin{proof}
Write $\cR=\End_\cD(W)$ where $\cD$ is a graded division algebra and $W$ is a right vector space over $\cD$. Then $W_0=\veps W$ is a graded $\cD$-subspace and $\veps\cR\veps$ can be identified with $\End_\cD(W_0)$. The result follows.
\end{proof}

Note that the Brauer class $[\cR]$ alone does not determine $\cR$ as a graded algebra, but it does so in conjunction with another invariant: the $G$-orbit of the multiset (i.e., a set whose elements are assigned multiplicity) $\{g_1T,\ldots,g_kT\}$ in $G/T$, where $G$ acts on $G/T$ by translations \cite{BK10}.

To conclude this section, we make an observation that will be useful later. Any group homomorphism $f\colon G\to G'$ yields a functor $F$ from the category of $G$-graded vector spaces (algebras, etc.) to that of $G'$-graded vector spaces (algebras, etc.). Namely, $F$ sends $V=\bigoplus_{g\in G}V_g$ to $V=\bigoplus_{g'\in G'}V_{g'}$, where $V_{g'}=\bigoplus_{g\in f^{-1}(g')} V_g$, and is identical on morphisms. Since any $G$-graded algebra can be regarded as a $G'$-graded algebra in this way, we obtain a homomorphism from the $G$-graded Brauer group to the $G'$-graded Brauer group (i.e., the graded Brauer group is a covariant functor in $G$). In the realization of the graded Brauer group of an algebraically closed field in terms of bicharacters, this homomorphism simply maps $\hat{\beta}$ to $\hat\gamma=\hat\beta\circ(\hat{f}\times\hat{f})$ where $\hat{f}\colon\wh{G'_0}\to\wh{G_0}$ is induced by the restriction $f\colon G_0\to G'_0$. The nondegenerate bicharacter $\gamma$ corresponding to $\hat\gamma$ can be calculated in terms of $\beta$ using the following fact.

\begin{lemma}\label{lm:beta_for_quotient}
Let $T$ be a finite abelian group and let $\beta$ be a symmetric or skew-symmetric bicharacter on $T$. Assume that $\beta$ is nondegenerate, so $t\mapsto\beta(t,\cdot)$ is an isomorphism $T\to\wh{T}$. Let $\hat{\beta}$ be the nondegenerate bicharacter on $\wh{T}$ obtained by composing $\beta$ with the inverse of this isomorphism. Given a subgroup $H\subset T$, consider the restriction $\hat\gamma$ of $\hat\beta$ to the subgroup $H^\perp\subset\wh{T}$. Then $(\rad\hat\gamma)^\perp$ is the subgroup $HH'\subset T$ where
\[
H'=\{t\in T\;|\;\beta(t,h)=1\mbox{ for all }h\in H\}.
\]
Moreover, $\beta$ induces a nondegenerate bicharacter $\gamma$ on $HH'/H$, which correponds to $\hat\gamma$, as follows:  $\gamma(xH,yH)=\beta(x,y)$ for all $x,y\in HH'$.\qed
\end{lemma}

\section{Clifford theory for graded modules}\label{s:Clifford_theory}

Let $\cL$ be a semisimple finite-dimensional Lie algebra over an algebraically closed field $\FF$ of characteristic $0$. Given a grading $\cL=\bigoplus_{g\in G}\cL_g$ by an abelian group $G$, we want to classify (finite-dimensional) graded-simple $\cL$-modules. When working with a specific graded $\cL$-module $W$, we may replace $G$ with the subgroup generated by the supports of $\cL$ and $W$. Thus we may assume, without loss of generality, that $G$ is {\em finitely generated}.

The classical Clifford theory relates irreducible representations of a group and those of its normal subgroup of finite index. A similar approach can be used to describe simple modules over the smash product $\cA\#\FF H$ in terms of those of $\cA$, where $\cA$ is an associative algebra and $H$ is a finite group acting on $\cA$ by automorphisms (see e.g. \cite{RR02}). To study graded $\cL$-modules, we can take $\cA=U(\cL)$ and $H=\wh{G}$ but note that $H$ is not necessarily finite and hence should be regarded as an {\em algebraic group} (a quasitorus) in order to have equivalence between graded $\cL$-modules and modules over the smash product $U(\cL)\#\FF\wh{G}$. Therefore, we outline here the relevant version of Clifford theory, interpreting the results in the language of gradings. Some of these ideas already appeared in relation to graded modules in \cite{BL07}, where the representation theory of quantum tori was used instead of Clifford theory.

\subsection{Twisting modules by an automorphism}

Since $\cL$ is $G$-graded, the algebraic group $\wh{G}$ acts on $\cL$ by automorphisms. For $\chi\in\wh{G}$, we will denote by $\alpha_\chi$ the corresponding automorphism of $\cL$: $\alpha_\chi(x)=\chi*x$ for all $x\in\cL$ (see \eqref{eq:hat_action} for the definition of $*$). This automorphism extends uniquely to an automorphism of $U(\cL)$ (associated with the induced $G$-grading), which we will denote by $\alpha_\chi$ as well. If $W$ is a graded $\cL$-module then $\wh{G}$ also acts on $W$.  For $\chi\in\wh{G}$, we denote by $\vphi_\chi$ the corresponding linear transformation of $W$: $\vphi_\chi(w)=\chi*w$ for all $w\in W$. The condition $\cL_g W_h\subset W_{gh}$ translates to the following:
\begin{equation}\label{eq:alpha_phi}
\vphi_\chi(xw)=\alpha_\chi(x)\vphi_\chi(w)\quad\mbox{for all}\;\chi\in\wh{G}, x\in\cL, w\in W.
\end{equation}

For any $\cL$-module $V$ and $\alpha\in\Aut(\cL)$, we denote by $V^\alpha$ the corresponding twisted $\cL$-module, i.e., the vector space $V$ with a different $\cL$-action: $x\cdot v=\alpha(x)v$ for all $x\in\cL$ and $v\in V$. In other words, if $V$ is an $\cL$-module via $\rho\colon\cL\to\Gl(V)$ (or, equivalently, via $\rho\colon U(\cL)\to\End(V)$) then $V^\alpha$ is an $\cL$-module via $\rho\circ\alpha$. Clearly, if $\psi\colon V_1\to V_2$ is a homomorphism of $\cL$-modules then $\psi$ is also a homomorphism $V_1^\alpha\to V_2^\alpha$. Hence the group $\Aut(\cL)$ acts (on the right) on the set of isomorphism classes of $\cL$-modules. It is well known that, for any inner automorphism $\alpha$ and any $\cL$-module $V$, the twisted module $V^\alpha$ is isomorphic to $V$, so the action of $\Aut(\cL)$ on the isomorphism classes of $\cL$-modules factors through the quotient group $\mathrm{Out}(\cL)\bydef\Aut(\cL)/\inaut(\cL)$. In particular, the orbits are finite. Moreover, fixing a Cartan subalgebra of $\cL$ and a system of simple roots, we can split the quotient map $\Aut(\cL)\to\mathrm{Out}(\cL)$ using the subgroup of diagram automorphisms. Hence the action of $\Aut(\cL)$ on the isomorphism classes of simple $\cL$-modules can be represented as an action of $\mathrm{Out}(\cL)$ on dominant integral weights by permuting the vertices of the Dynkin diagram and the corresponding fundamental weights $\omega_1,\ldots,\omega_r$, where $r=\rank\cL$. Explicitly, an element $\alpha\in\Aut(\cL)$ can be uniquely written in the form $\alpha=\alpha_0\tau$ where $\alpha_0\in\inaut(\cL)$ and $\tau$ is a diagram automorphism, so $V_\lambda^\alpha\cong V_{\tau^{-1}(\lambda)}$ for any dominant integral weight $\lambda$, where $V_\lambda$ is the simple module of highest weight $\lambda$.

In particular, any $\chi\in\wh{G}$ can be used to twist an $\cL$-module $V$ through $\alpha=\alpha_\chi$. We will write $V^\chi$ for $V^{\alpha_\chi}$. Now observe that, for a graded $\cL$-module $W$, equation \eqref{eq:alpha_phi} can be restated by saying that $\vphi_\chi$ is an isomorphism $W\to W^\chi$. Disregarding the grading of $W$, we can write $W$ as the direct sum of its isotypic components:
\[
W=\bigoplus_\lambda W_\lambda
\]
where $\lambda$ ranges over a finite set of dominant integral weights of $\cL$ and $W_\lambda$ is a sum of copies of $V_\lambda$. It follows that $\vphi_\chi$ maps $W_\lambda$ onto $W_\mu$ where $\mu$ is determined by the condition $V_\mu\cong V_\lambda^{\chi^{-1}}$, so $\mu=\tau_\chi(\lambda)$ where $\tau_\chi$ is the diagram automorphism representing the class of $\alpha_\chi$ in $\mathrm{Out}(\cL)$. We denote by $K_\lambda$ the {\em inertia group} of $\lambda$:
\[
K_\lambda=\{\chi\in\wh{G}\;|\;\tau_\chi(\lambda)=\lambda\}=\{\chi\in\wh{G}\;|\;V_\lambda^\chi\cong V_\lambda\}.
\]
Note that, since $\chi\mapsto\alpha_\chi$ is a homomorphism of algebraic groups $\wh{G}\to\Aut(\cL)$ and the stabilizer of $\lambda$ in $\Aut(\cL)$ is Zariski closed (as it contains the closed subgroup $\inaut(\cL)$ of finite index), the subgroup $K_\lambda$ is Zariski closed in $\wh{G}$. Let $H_\lambda=K_\lambda^\perp\subset G$, so $H_\lambda^\perp=K_\lambda$ and $|H_\lambda|$ equals the size of the orbit $\wh{G}\lambda$. Since $\sum_{\mu\in\wh{G}\lambda}W_\mu$ is a $\wh{G}$-invariant (equivalently, $G$-graded) $\cL$-submodule of $W$, we conclude that all graded-simple $\cL$-modules have the form $V_{\lambda_1}^k\oplus\cdots\oplus V_{\lambda_s}^k$ where $\{\lambda_1,\ldots,\lambda_s\}$ is a $\wh{G}$-orbit and $k$ is a positive integer. We will show that, for any orbit, this $k$ is uniquely determined.

\subsection{Brauer invariant and graded Schur index of a simple module}

Let $V=V_\lambda$ be the simple module of highest weight $\lambda$ and consider the corresponding homomorphism $\rho\colon U(\cL)\to\End(V)$. By Density Theorem, this is a surjection. For any $\chi\in K_\lambda$, we have $V^\chi\cong V$, which means that there exists $u_\chi\in\GL(V)$ such that $\rho(\alpha_\chi(a))=u_\chi\rho(a)u_\chi^{-1}$ for all $a\in U(\cL)$. Clearly, the inner automorphism $\wt{\alpha}_\chi(x)=u_\chi x u_\chi^{-1}$ of $\End(V)$ is uniquely determined, so the operator $u_\chi$ is determined up to a scalar multiple. Thus we obtain a representation $K_\lambda\to\Aut(\End(V))$ given by $\chi\mapsto\wt{\alpha}_\chi$. In particular, the operators $u_\chi$, $\chi\in K_\lambda$, commute up to a scalar. Note that we can select a basis in $\End(V)$ consisting of the images of homogeneous elements of $U(\cL)$ (with respect to the $G$-grading), so the representation $K_\lambda\to\Aut(\End(V))$ is a homomorphism of algebraic groups and corresponds to a $\bG$-grading on $\End(V)$ where $\bG=G/H_\lambda$. This is the unique $\bG$-grading on $\End(V)$ such that $\rho\colon U(\cL)\to\End(V)$ is a homomorphism of graded algebras. Here $U(\cL)$ is considered with the $\bG$-grading induced by the quotient map $G\to\bG$ (which is a coarsening of the $G$-grading), and the existence of the mentioned $\bG$-grading on $\End(V)$ can also be seen from the fact that the kernel of $\rho$ is a $\bG$-graded ideal.

\begin{df}
The class $[\End(V_\lambda)]$ in the $(G/H_\lambda)$-graded Brauer group will be called the {\em Brauer invariant} of $\lambda$ (or of $V_\lambda$) and denoted by $\Br(\lambda)$. The degree of the graded division algebra $\cD$ representing $\Br(\lambda)$ will be called the {\em (graded) Schur index} of $\lambda$ (or of $V_\lambda$).
\end{df}

Recalling the description of the graded Brauer group from the previous section, we see that $\Br(\lambda)$ is determined by the ``commutation factor'' of the operators $u_\chi$ ($\chi\in K_\lambda$), i.e., the alternating bicharacter $\hat\beta$ of $K_\lambda$ given by \eqref{eq:characterization_hat_beta}.

\begin{proposition}\label{prop:Schur_index}
Let $V=V_\lambda$ and $\bG=G/H_\lambda$. The $\cL$-module $V^k$ admits a $\bG$-grading that makes it a graded-simple $\cL$-module if and only if $k$ equals the Schur index of $V$. This $\bG$-grading is unique up to isomorphism and shift.
\end{proposition}

\begin{proof}
If $W=V^k$ admits a $\bG$-grading making it a graded $\cL$-module then $\End(W)$ has an induced elementary grading and $\rho^{\oplus k}\colon U(\cL)\to\End(W)$ is a homomorphism of $\bG$-graded algebras. Since $\rho^{\oplus k}$ has the same kernel as $\rho$, we see that $\End(W)$ contains a $\bG$-graded unital subalgebra $\cR$ isomorphic to $\End(V)$. Let $\cC$ be the centralizer of $\cR$ in $\End(W)$. Then $\cC$ is $\bG$-graded and $\End(W)\cong\cR\ot\cC$ as $\bG$-graded algebras. Hence $[\cR]=[\cC]^{-1}$ in the $\bG$-graded Brauer group. If $W$ is graded-simple then $\cC$ is a graded division algebra, so $\cC\cong\cD^\op$ where $\cD$ is the graded division algebra representing $[\cR]=\Br(\lambda)$. Hence $k^2=\dim\cC=\dim\cD$, so $k$ equals the Schur index of $\lambda$. (If $W$ is $\bG$-graded but not necessarily graded-simple then $k$ must be a multiple of the Schur index.)

Conversely, suppose $k$ equals the Schur index. Again, write $\End(W)\cong\cR\ot\cC$ (as algebras) where $\cR$ is the image of $U(\cL)$ and $\cC$ is its centralizer, which is a matrix algebra of degree $k$ and hence can be given a $\bG$-grading using any isomorphism of algebras $\cD^\op\to\cC$. Then $\End(W)$ also receives a $\bG$-grading, which is elementary since $[\cR][\cC]$ is trivial in the graded Brauer group. Therefore, $W$ has a $\bG$-grading that induces the grading of $\End(W)$. Since $\cC$ is a graded division algebra, $W$ is graded-simple as an $\cR$-module. Since $\cR$ is a graded-simple algebra, it has a unique graded-simple module, up to isomorphism and shift of grading.
\end{proof}

\begin{remark}\label{rem:shift}
Under the conditions of Proposition \ref{prop:Schur_index}, let $\ell$ be the Schur index of $V$. Fix a $\bG$-grading on $U=V^\ell$ as in Proposition \ref{prop:Schur_index} and pick an element $\bg\in\bG$. Since the action of $\cL$ factors through $\End(V)$, the shift $U^{[\bg]}$ is isomorphic to $U$ as a graded $\cL$-module if and only if $\bg$ belongs to the support of the graded division algebra representing $[\End(V)]$.
\end{remark}

\subsection{Induced graded space}

To classify graded-simple $\cL$-modules up to isomorphism, we need another ingredient: the construction of a graded space induced from a quotient group. Let $H$ be a finite subgroup of $G$ and let $U=\bigoplus_{\bg\in\bG}U_{\bg}$ be a $\bG$-graded vector space where $\bG=G/H$. Let $K=H^\perp\subset\wh{G}$. Then $U$ is a $K$-module, so we can consider $W=\Ind_K^{\wh{G}}U\bydef\FF\wh{G}\ot_{\FF K}U$, which is a $\wh{G}$-module and thus a $G$-graded space. Note that, as a $K$-module or $\bG$-graded space, $W$ is just the sum of $s$ copies of $U$, where $s=|H|$. The $G$-grading on $W$ can be explicitly described as follows. Take $\wh{H}=\{\chi_1,\ldots,\chi_s\}$ and extend each $\chi_i$ to a character of $G$ in some way. We will denote the extensions by $\chi_i$ as well, and we may assume that $\chi_1$ is the trivial character. Then $\{\chi_1,\ldots,\chi_s\}$ is a transversal for the subgroup $K$ in $\wh{G}$ and hence
\[
\Ind_K^{\wh{G}}U=\chi_1\ot U\oplus\cdots\oplus\chi_s\ot U.
\]
One verifies that, for any $g\in G$, the homogeneous component $W_g$ is given by
\[
W_g=\big\{\sum_{j=1}^s\chi_j\ot\chi_j(g)^{-1}u\;|\;u\in U_{\bg}\big\}.
\]
We will denote the $G$-graded space $W=\bigoplus_{g\in G}W_g$ by $\Ind_{G/H}^G U$.

If $U$ is a $\bG$-graded $\cL$-module then $\cL$ acts on $W$ as follows: $x\cdot(\chi\ot u)=\chi\ot\alpha_{\chi^{-1}}(x)u$ for all $x\in\cL$, $\chi\in\wh{G}$ and $u\in U$. One verifies that this action is well defined and turns $W$ into a $G$-graded $\cL$-module. Moreover, if $U$ is simple as a $\bG$-graded $\cL$-module and $K$ is its inertia group then $W$ is simple as a $G$-graded $\cL$-module. Indeed, let $W_0\subset W$ be a nonzero $G$-graded $\cL$-submodule. Since $\chi_i\ot U\cong U^{\chi_i^{-1}}$ are simple as $\bG$-graded $\cL$-modules and pairwise non-isomorphic, we conclude that there exists $i$ such that $\chi_i\ot U\subset W_0$. But $\chi_i\ot U$ generates $W$ as a $\wh{G}$-module, hence $W_0=W$.

\subsection{Classification of graded-simple modules}

Let $\lambda$ be a dominant integral weight. If $\mu\in\wh{G}\lambda$ then $H_\lambda=H_\mu$ (since $\wh{G}$ is abelian) and the $(G/H_\lambda)$-graded algebras $\End(V_\lambda)$ and $\End(V_\mu)$ are isomorphic (with an isomorphism induced by $\alpha_\chi\colon U(\cL)\to U(\cL)$ for some $\chi\in\wh{G}$). Hence $\lambda$ and $\mu$ have the same graded Schur index and Brauer invariant.

\begin{df}\label{df:WO}
For each $\wh{G}$-orbit $\cO$ in the set of dominant integral weights, we select a representative $\lambda$ and equip $U=V_\lambda^\ell$ with a $(G/H_\lambda)$-grading as in Proposition \ref{prop:Schur_index}, where $\ell$ is the graded Schur index of $V_\lambda$. Then $W(\cO)\bydef\Ind_{G/H_\lambda}^G U$ is a graded-simple $\cL$-module.
\end{df}

\begin{theorem}\label{th:graded_simple_modules}
Let $\cL$ be a semisimple finite-dimensional Lie algebra over an algebraically closed field of characteristic $0$. Suppose $\cL$ is graded by an abelian group $G$. Then, for any graded-simple finite-dimensional $\cL$-module $W$, there exist a $\wh{G}$-orbit $\cO$ of dominant integral weights  and an element $g\in G$ such that $W$ is isomorphic to $W(\cO)^{[g]}$, with $W(\cO)$ as in Definition \ref{df:WO}. Moreover, two such graded modules, $W(\cO)^{[g]}$ and $W(\cO')^{[g']}$, are isomorphic if and only if $\cO'=\cO$ and $g' G_\lambda=g G_\lambda$ where $G_\lambda$ is the pre-image of the support of the Brauer invariant $[\End(V_\lambda)]$ under the quotient map $G\to G/H_\lambda$, with $\lambda$ being a representative of $\cO$.
\end{theorem}

\begin{proof}
We already know that, as an ungraded $\cL$-module, $W$ can be written as $W_1\oplus\cdots\oplus W_s$ where $W_i\cong V_{\lambda_i}^k$ and $\cO=\{\lambda_1,\ldots,\lambda_s\}$ is a $\wh{G}$-orbit. Let $\lambda=\lambda_1$ be the representative selected for $\cO$ in Definition \ref{df:WO} and let $K_\lambda=H_\lambda^\perp\subset\wh{G}$. Then the isotypic component $W_1$ is invariant under all $\vphi_\chi$, $\chi\in K_\lambda$, so $W_1$ is a $\bG$-graded $\cL$-module for $\bG=G/H_\lambda$. Let $\{\chi_1,\ldots,\chi_s\}$ be a transversal for $K_\lambda$ in $\wh{G}$ where $\chi_1$ is the trivial character. We order the $\chi_i$ in such a way that $\lambda_i=\tau_{\chi_i}(\lambda)$ and hence $\vphi_{\chi_i}$ maps $W_1$ onto $W_i$. For any $\bG$-graded $\cL$-submodule $W_0\subset W_1$, the direct sum $\sum_i\vphi_{\chi_i}(W_0)$ is a $G$-graded $\cL$-submodule of $W$. It follows that $W_1$ is graded-simple and hence $k=\ell$ by Proposition \ref{prop:Schur_index}. Moreover, there exists $g\in G$ such that $U^{[\bg]}$ is isomorphic to $W_1$ as a $\bG$-graded $\cL$-module. We fix an isomorphism $\vphi_1\colon U^{[\bg]}\to W_1$ and set $\vphi_i=\vphi_{\chi_i}\circ\vphi_1$. Then $\vphi_i\colon U^{[\bg]}\to W_i^{\chi_i}$ is an isomorphism of $\bG$-graded $\cL$-modules. Define a linear isomorphism $\vphi\colon W(\cO)^{[g]}\to W$ by setting $\vphi(\chi_i\ot u)=\chi_i(g)^{-1}\vphi_i(u)$. We claim that $\vphi$ is an isomorphism of $G$-graded $\cL$-modules. It follows from the definition of $\cL$-action on $W(\cO)$ that $\vphi$ is an isomorphism of $\cL$-modules. It remains to show that $\vphi$ is $\wh{G}$-equivariant.

Recall that $\chi*w=\vphi_\chi(w)$ for all $w\in W$ and $\chi\in\wh{G}$. Now fix $i$ and $\chi$. Then there exist unique $j$ and $\chi_0\in K_\lambda$ such that $\chi\chi_i=\chi_j\chi_0$. Since $\chi_0$ sends $H_\lambda$ to $1$, we regard it as a character of $\bG$ with $\chi_0(g)=\chi_0(\bg)$. By definition of $\vphi_1$, we have
\begin{equation}\label{eq:1}
\vphi_1(\chi_0(g)(\chi_0*u))=\chi_0*\vphi_1(u)=\vphi_{\chi_0}(\vphi_1(u))
\end{equation}
for any $u\in U$. By definition of induced module, we have
\[
\chi*(\chi_i\ot u)=\chi\chi_i\ot u=\chi_j\chi_0\ot u=\chi_j\ot\chi_0*u.
\]
Hence we compute, using the definition of $\vphi$ and equation \eqref{eq:1}:
\[
\begin{split}
\vphi\big(\chi(g)\chi*(\chi_i\ot u)\big)=&\chi_j(g)^{-1}\chi(g)\vphi_j(\chi_0*u)\\
=&\chi_i(g)^{-1}\vphi_j\big(\chi_0(g)(\chi_0*u)\big)\\
=&\chi_i(g)^{-1}\vphi_{\chi_j}\big(\vphi_1(\chi_0(g)(\chi_0*u))\big)\\
=&\chi_i(g)^{-1}\vphi_{\chi_j}\big(\vphi_{\chi_0}(\vphi_1(u))\big)\\
=&\chi_i(g)^{-1}\vphi_{\chi}\big(\vphi_{\chi_i}(\vphi_1(u))\big)\\
=&\chi_i(g)^{-1}\vphi_{\chi}\big(\vphi_i(u)\big)=\chi*\vphi(\chi_i\ot u),
\end{split}
\]
which is exactly what is required for $\vphi\colon W(\cO)^{[g]}\to W$ to be $\wh{G}$-equivariant.

If $g\in G_\lambda$ then $\bg$ is in the support of the graded division algebra representing $[\End(V_\lambda)]$ and hence, by Remark \ref{rem:shift}, the $\bG$-graded $\cL$-module $U^{[\bg]}$ is isomorphic to $U$.  Then the above argument allows us to extend an isomorphism $\vphi_1\colon U^{[\bg]}\to U$ to an isomorphism $\vphi\colon W(\cO)^{[g]}\to W(\cO)$.
Conversely, suppose that $\vphi\colon W(\cO)^{[g]}\to W(\cO)$ is an isomorphism of $G$-graded $\cL$-modules. Looking at the isotypic components as $\cL$-modules (disregarding the grading), we see that $\vphi$ must map $\chi_1\ot U$ onto $\chi_1\ot U$, so it restricts to an isomorphism $(\chi_1\ot U)^{[\bg]}\to\chi_1\ot U$ of $\bG$-graded $\cL$-modules. By Remark \ref{rem:shift}, $\bg$ must be in the support of the graded division algebra representing $[\End(V_\lambda)]$, so $g\in G_\lambda$.
\end{proof}

\begin{corollary}
An $\cL$-module $V$ admits a $G$-grading that would make it a graded $\cL$-module if and only if, for any dominant integral weight $\lambda$, the multiplicities of $V_{\lambda_1},\ldots,V_{\lambda_s}$ in $V$ are equal to each other and divisible by the graded Schur index of $\lambda$, where $\{\lambda_1,\ldots,\lambda_s\}$ is the $\wh{G}$-orbit of $\lambda$.\qed
\end{corollary}

\subsection{Comparison of graded modules for two isomorphic gradings on $\cL$}

Gradings by abelian groups have been classified up to isomorphism \cite{BK10,EKmon} for the classical simple Lie algebras except $D_4$. Let us compare the theories of graded modules for two isomorphic $G$-gradings on a semisimple Lie algebra $\cL$. So let $\Gamma:\cL=\bigoplus_{g\in G}\cL_g$ and $\Gamma':\cL=\bigoplus_{g\in G}\cL'_g$ be two such gradings and let $\alpha$ be an automorphism of $\cL$ such that $\alpha(\cL_g)=\cL'_g$ for all $g\in G$. We will use prime to distinguish the objects associated to $\Gamma'$ from those associated to $\Gamma$. For any $\chi\in\wh{G}$, we have $\alpha'_\chi=\alpha\alpha_\chi\alpha^{-1}$. Suppose $V_\lambda\cong V_\mu^\alpha$. Then
\[
K_\lambda=\{\chi\in\wh{G}\;|\;V_\lambda^{\alpha_\chi}\cong V_\lambda\}
=\{\chi\in\wh{G}\;|\;V_\mu^{\alpha\alpha_\chi}\cong V_\mu^\alpha\}
=\{\chi\in\wh{G}\;|\;V_\mu^{\alpha'_\chi}\cong V_\mu\}=K'_\mu.
\]
Since $\alpha$ maps the kernel of $\rho_\lambda\colon U(\cL)\to\End(V_\lambda)$ onto the kernel of $\rho_\mu\colon U(\cL)\to\End(V_\mu)$, we obtain a commutative diagram
\[
\xymatrix{
{U(\cL)}\ar[r]^\alpha\ar[d]_{\rho_\lambda} & U(\cL)\ar[d]^{\rho_\mu}\\
{\End(V_\lambda)}\ar[r]^{\wt{\alpha}} & \End(V_\mu)
}
\]
where the vertical arrows are surjective and $\wt{\alpha}$ is an isomorphism of algebras. Clearly, $\wt{\alpha}$ maps the $\bG$-grading of $\End(V_\lambda)$ induced by $\Gamma$ to the $\bG$-grading of $\End(V_\mu)$ induced by $\Gamma'$, where $\bG=G/H_\lambda=G/H'_\mu$. Hence $\Br(\lambda)=\Br'(\mu)$. We conclude that the theory of graded modules for $\Gamma'$ can be obtained from that for $\Gamma$ as long as an isomorphism $\alpha$ is known. In particular, if $\alpha$ is inner then we have $H'_\lambda=H_\lambda$ and $\Br'(\lambda)=\Br(\lambda)$ for all $\lambda$.

For the classical simple Lie algebras of type $A_1$ and of series $B$ and $C$, all automorphisms are inner so the $H_\lambda$ are trivial, and it suffices to calculate $\Br(\lambda)$ for one representative of each isomorphism class of gradings. For series $A$ (except $A_1$) and $D$ (except $D_4$), the index of $\inaut(\cL)$ in $\Aut(\cL)$ is $2$ so the $H_\lambda$ have order at most $2$, and each isomorphism class of gradings consists of at most $2$ orbits under the action of $\inaut(\cL)$. Clearly, the isomorphism class of a grading $\Gamma$ is one $\inaut(\cL)$-orbit if and only if $\Stab(\Gamma)$ (i.e., the group of automorphisms of $\cL$ as a $G$-graded algebra) contains an outer automorphism. This is definitely the case if $\Gamma$ is an ``outer grading'' (i.e., the image of $\bG$ in $\Aut(\cL)$ contains an outer automorphism). For series $A$, the gradings are already classified in \cite{BK10} as inner (Type I) and outer (Type II), and it is clear when two inner gradings belong to one $\inaut(\cL)$-orbit. For series $D$, this will be done in Section \ref{s:D}. It should be noted that type $D_4$ is no different from other members of the series if we restrict our attention to ``matrix gradings'', i.e., those that can be obtained by restricting a grading of $M_8(\FF)$ (or, equivalently, such that $\wh{G}$ fixes the isomorphism class of the natural module).

\subsection{Calculation of Brauer invariants}

We conclude this section by showing how, for simple Lie algebras, the calculation of Brauer invariants of all dominant integral weights  can be reduced to the fundamental weights $\omega_1,\ldots,\omega_r$.

\begin{proposition}\label{prop:product}
Let $\lambda_1$ and $\lambda_2$ be dominant integral weights of a semisimple Lie algebra $\cL$ and let $\mu=\lambda_1+\lambda_2$. Suppose $\cL$ is equipped with a grading by an abelian group $G$ such that $H_{\lambda_1}\subset H_{\mu}$ (equivalently, $H_{\lambda_2}\subset H_{\mu}$). Then $\Br(\mu)=\Br(\lambda_1)\Br(\lambda_2)$ in the $(G/H_\mu)$-graded Brauer group.
\end{proposition}

\begin{proof}
Denote $\cR_i=\End(V_{\lambda_i})$ ($i=1,2$) and $\bG=G/H_\mu$. Let $V=V_{\lambda_1}\ot V_{\lambda_2}$ and $\cR=\End(V)$. By definition of $H_{\lambda_i}$, the algebra $\cR_i$ is $(G/H_{\lambda_i})$-graded in such a way that the surjection $\rho_i\colon U(\cL)\to\cR_i$ is a  homomorphism of $(G/H_{\lambda_i})$-algebras, hence a fortiori a homomorphism of $\bG$-graded algebras. Let $\rho\colon U(\cL)\to\cR$ be the tensor product representation, i.e., the representation of $\cL$ on $V$ given by $\rho(x)=\rho_1(x)\ot 1+1\ot\rho_2(x)$ for all $x\in\cL$. If we give $\cR$ a $\bG$-grading by identifying $\cR$ and $\cR_1\ot\cR_2$ in the natural way then $\rho\colon U(\cL)\to\cR$ is a homomorphism of $\bG$-graded algebras. This homomorphism is not surjective, in general, so let $\cA$ be its image, which is  a graded unital subalgebra of $\cR$. It is known that the simple module $V_\mu$ occurs in $V$ with multiplicity $1$. Let $\veps\in\cR$ be the projection of $V$ onto $V_\mu$ associated with the decomposition of $V$ into isotypic components (as an $\cL$-module). Then $\veps$ is the identity element in the block $\cA_0$ of the semisimple algebra $\cA$ associated to the $\cA$-module $V_\mu$. By definition of $H_\mu$, the kernel $I$ of the surjection $\rho\colon U(\cL)\to\End(V_\mu)$ is a $\bG$-graded ideal of $U(\cL)$, hence $\rho(I)$ is a graded ideal of $\cA$. But $\cA_0$ is the annihilator of $\rho(I)$ in $\cA$, so $\cA_0$ is a graded subalgebra of $\cR$. Moreover, $\cA_0\cong\End(V_\mu)$ as $\bG$-graded algebras. Therefore,
\[
\Br(\mu)=[\End(V_\mu)]=[\cA_0]=[\veps\cR\veps]=[\cR]=[\cR_1][\cR_2]=\Br(\lambda_1)\Br(\lambda_2),
\]
where we have used Lemma \ref{lm:corner}.
\end{proof}

If $\cL$ is simple then $\mathrm{Out}(\cL)$ has order at most $2$ in all cases except $D_4$. For $D_4$, we have $\mathrm{Out}(\cL)\cong S_3$ so the image of the abelian group $\wh{G}$ in $\mathrm{Out}(\cL)$ has order at most $3$. It follows that, in all cases, the $\wh{G}$-orbits in the set of dominant integral weights have length at most $3$, and $\wh{G}$ acts cyclically on each orbit.  Consider $\lambda=\sum_{i=1}^s m_i\omega_{k_i}$ where the $\{\omega_{k_1},\ldots,\omega_{k_s}\}$ is an orbit and $m_i$ are nonnegative integers. Since $\wh{G}$ cyclically permutes the $\omega_{k_i}$, we have one of the following two possibilities: either all $m_i$ are equal or $H_\lambda=H_{\omega_{k_i}}$ (for any $i$). It follows that, if we know the Brauer invariant of $\sum_{i=1}^s\omega_{k_i}$ for each orbit, then we can compute it for any dominant integral weight using Proposition \ref{prop:product}.

\begin{proposition}\label{prop:orbit_size_2}
Let $\lambda_1$ and $\lambda_2$ be dominant integral weights of a semisimple Lie algebra $\cL$ and let $\mu=\lambda_1+\lambda_2$. Suppose $\cL$ is  equipped with a grading by an abelian group $G$ such that $\{\lambda_1,\lambda_2\}$ is a $\wh{G}$-orbit, so $H_\mu=\{e\}$ and $H_{\lambda_i}=\langle h\rangle$, $i=1,2$, where $h$ has order $2$, and the algebra $\cR=\End(V_{\lambda_1}\ot V_{\lambda_2})$ has a $\bG$-grading by identification with $\End(V_{\lambda_1})\ot\End(V_{\lambda_2})$, where $\bG=G/\langle h\rangle$. Fix $\chi\in\wh{G}$ such that $\chi(h)=-1$ and pick isomorphisms $u'\colon V_{\lambda_1}\to V_{\lambda_2}^\chi$ and $u''\colon V_{\lambda_2}\to  V_{\lambda_1}^\chi$. Then $\Br(\mu)=[\cR]$ in the $G$-graded Brauer group where the $G$-grading on $\cR$ is obtained by refining the $\bG$-grading as follows:
\[
\cR_g=\{x\in\cR_{\bg}\;|\;uxu^{-1}=\chi(g)x\}\quad\mbox{for all}\; g\in G,
\]
where $u=(u''\ot u')\circ\tau$ and $\tau$ is the flip $v'\ot v''\mapsto v''\ot v'$ for all $v'\in V_{\lambda_1}$ and $v''\in V_{\lambda_2}$ (so $u\in\cR$).
\end{proposition}

\begin{proof}
Let $\cR_i=\End(V_{\lambda_i})$. Then $\rho_i\colon U(\cL)\to\cR_i$ are homomorphisms of $\bG$-graded algebras and the following diagram commutes:
\[
\xymatrix{
{U(\cL)}\ar[r]^{\alpha_\chi}\ar[d]_{\rho_1} & U(\cL)\ar[d]^{\rho_2}\ar[r]^{\alpha_\chi} & {U(\cL)}\ar[d]^{\rho_1}\\
{\cR_1}\ar[r]^{\alpha'} & {\cR_2}\ar[r]^{\alpha''} & {\cR_1}
}
\]
where $\alpha'(a)=u'a(u')^{-1}$ for all $a\in\cR_1$ and $\alpha''(b)=u''b(u'')^{-1}$ for all $b\in\cR_2$. Since $\alpha_\chi$ is an isomorphism of $\bG$-graded algebras, so are $\alpha'$ and $\alpha''$. Define an automorphism $\wt{\alpha}\colon\cR\to\cR$ (as a $\bG$-graded algebra) by setting
\[
\wt{\alpha}(a\ot b)=\alpha''(b)\ot\alpha'(a)\quad\mbox{for all}\;a\in\cR_1, b\in\cR_2.
\]
Then $\wt{\alpha}^2(a\ot b)=\alpha''\alpha'(a)\ot\alpha'\alpha''(b)=(\chi^2*a)\ot(\chi^2*b)$ where, as usual, $*$ denotes the actions of $K=K_{\lambda_1}=K_{\lambda_2}$ associated to the $\bG$-gradings. Hence $\wt{\alpha}^2$ acts as the scalar operator $\chi^2(\bg)$ on the homogeneous component $\cR_{\bg}$ (where we regard $\chi^2$ as a character of $\bG$), and we obtain a $G$-grading on $\cR$ by setting
\[
\cR_g=\{x\in\cR_{\bg}\;|\;\wt{\alpha}(x)=\chi(g)x\}\quad\mbox{for all}\; g\in G.
\]
For $\rho=\rho_1\ot\rho_2$, one checks that $\rho(\alpha_\chi(x))=\wt{\alpha}(\rho(x))$ for all $x\in\cL$. It follows that $\rho\colon U(\cL)\to\cR$ is a homomorphism of $G$-graded algebras. By the same argument as in the proof of Proposition \ref{prop:product}, we obtain $\Br(\mu)=[\cR]$ in the $G$-graded Brauer group ($H_\mu$ is trivial in our case). It remains to observe that
\[
\wt\alpha(a\ot b)=u''b(u'')^{-1}\ot u'a(u')^{-1}=(u''\ot u')\tau(a\ot b)\tau(u''\ot u')^{-1}=u(a\ot b)u^{-1}
\]
for all $a\in\cR_1$ and $b\in\cR_2$.
\end{proof}

The result is especially simple in the case when $V_{\lambda_1}$ and $V_{\lambda_2}$ are dual to each other.

\begin{proposition}\label{prop:product_with_dual}
Let $\lambda_1$ and $\lambda_2$ be dominant integral weights of a semisimple Lie algebra $\cL$ such that $V_{\lambda_2}\cong V_{\lambda_1}^*$, and let $\mu=\lambda_1+\lambda_2$. Suppose $\cL$ is  equipped with a grading by an abelian group $G$ such that $\{\lambda_1,\lambda_2\}$ is a $\wh{G}$-orbit. Then $\Br(\mu)$ is trivial.
\end{proposition}

\begin{proof}
We will use the notation of Proposition \ref{prop:orbit_size_2} and the abbreviations  $V=V_{\lambda_1}$ and $V^*=V_{\lambda_2}$. Then we have the following commutative diagram:
\[
\xymatrix{
{\cL^\op}\ar[r]^{-\id}\ar[d]_{\rho_1} & {\cL}\ar[d]^{\rho_2}\\
{\End(V)^\op}\ar[r]^{\mathrm{adj}} & \End(V^*)
}
\]
where $\mathrm{adj}$ denotes the map sending an operator $a\in\End(V)$ to its adjoint $a^*\in\End(V^*)$, which is determined by
\[
\langle x,ay\rangle=\langle a^*x,y\rangle\quad\mbox{for all}\;x\in V^*,y\in V,
\]
where $\langle\cdot,\cdot\rangle$ is the canonical pairing between $V^*$ and $V$. It follows that $[\End(V^*)]$ is the inverse of $[\End(V)]$ in the $\bG$-graded Brauer group.  Hence the $\bG$-grading on $\cR=\End(V\ot V^*)$ is elementary. We claim that this elementary grading on $\cR$ is induced from the $\bG$-grading on the vector space $V\ot V^*$ coming from the natural isomorphism $V\ot V^*\cong\End(V)$. Indeed, for any $\psi\in K\bydef K_{\lambda_1}=K_{\lambda_2}$, we have $\rho_1(\alpha_\psi(x))=u_\psi\rho_1(x)u_\psi^{-1}$ for all $x\in\cL$, hence $\rho_2(\alpha_\psi(x))=(u_\psi^*)^{-1}\rho_2(x)u_\psi^*$ for all $x\in\cL$. It follows that the $\bG$-grading on the algebra $\cR=\cR_1\ot\cR_2$ is associated to the $K$-action $\psi*x=(u_\psi\ot (u_\psi^*)^{-1})x(u_\psi\ot (u_\psi^*)^{-1})^{-1}$. Under the natural isomorphism $V\ot V^*\cong\End(V)$, the operator $u_\psi\ot (u_\psi^*)^{-1}$ on $V\ot V^*$ corresponds to the inner automorphism $\Ad(u_\psi)$ of $\End(V)$, and these inner automorphisms determine the $\bG$-grading of $\End(V)$.

Now consider $\chi\in\wh{G}\setminus K$, i.e., $\chi\in\wh{G}$ with $\chi(h)=-1$. We may choose the associated isomorphisms $u'\colon V\to (V^*)^\chi$ and $u''\colon V^*\to V^\chi$ so that $u''=((u')^*)^{-1}$. Define a nondegenerate bilinear form on $V$ as follows:
\[
(x,y)=\langle u'x,y\rangle\quad\mbox{for all}\;x,y\in V,
\]
and then define a linear map $\vphi\colon\End(V)\to\End(V)$ by setting
\[
(ax,y)=(x,\vphi(a)y)\quad\mbox{for all}\;x,y\in V,\,a\in\End(V),
\]
i.e., $\vphi(a)$ is the adjoint of $a$ with respect to the bilinear form $(\cdot,\cdot)$. Clearly, $\vphi$ is an anti-automorphism. One checks that $\vphi(a)^*=u'a(u')^{-1}$ for all $a\in\End(V)$. Therefore, we have the following commutative diagram:
\[
\xymatrix{
{\cL^\op}\ar[r]^{-\alpha_\chi}\ar[d]_{\rho_1} & {\cL}\ar[d]^{\rho_1}\\
{\End(V)^\op}\ar[r]^{\vphi} & \End(V)
}
\]
Since $-\alpha_\chi$ is a homomorphism of $\bG$-graded algebras, so is $\vphi$. Moreover, $\vphi^2$ acts as the scalar operator $\chi^2(\bg)$ on the homogeneous component $\End(V)_{\bg}$. It follows that we obtain a $G$-grading on the vector space $\End(V)$ by setting
\begin{equation}\label{eq:phi_refinement}
\End(V)_g=\{a\in\End(V)_{\bg}\;|\;\vphi(a)=-\chi(g)a\}\quad\mbox{for all}\;g\in G.
\end{equation}
Since $-\vphi$ is a Lie homomorphism, this is actually a $G$-grading on the Lie algebra $\brac{\End(V)}$. By construction, $\rho_1\colon\cL\to\brac{\End(V)}$ is a homomorphism of $G$-graded algebras. Finally, we claim that the operator $u=(u''\ot u')\tau$ on $V\ot V^*$ corresponds to $\vphi$ on $\End(V)$ under the natural isomorphism $V\ot V^*\cong\End(V)$. Indeed, for any $z\in V$ and $f\in V^*$, the element $z\ot f$ corresponds to $a(\cdot)=f(\cdot)z\in\End(V)$. For any $x,y\in V$, we compute:
\[
\begin{split}
((z\ot f)x,y)&=(f(x)z,y)=\langle f,x \rangle\langle u'z,y\rangle=\langle u'x,u''f\rangle\langle u'z,y\rangle\\
&=(x,(u'z)(y)(u''f))=(x,u(z\ot f)(y)),
\end{split}
\]
proving the claim. Therefore, in our case, the $G$-grading on $\cR=\End(\End(V))$ in Proposition \ref{prop:orbit_size_2} is induced from the above $G$-grading on the vector space $\End(V)$. Thus $\Br(\mu)=[\cR]$ is trivial.
\end{proof}

\section{Series A}\label{s:A}

Consider the simple Lie algebra $\cL=\Sl_{r+1}(\FF)$ of type $A_{r}$, where $\FF$ is an algebraically closed field of characteristic $0$. Given a $G$-grading on $\cL$, we are going to find the Brauer invariant of every simple $\cL$-module. The parameters used in \cite{BK10} (see also \cite{EKmon}) to determine the $G$-grading on $\cL$ (up to isomorphism) include what are, in our present terminology, the inertia group and the Brauer invariant of the natural module of $\cL$.

\subsection{Preliminaries}

Recall that the calculation of Brauer invariants reduces to modules of the form $V_\lambda$ where $\lambda$ is the sum of fundamental weights forming a $\wh{G}$-orbit. The simple $\cL$-module of highest weight $\omega_i$ ($i=1,\ldots,r$) can be realized as $\wedge^i V$ where $V$ is the natural module of $\cL$, which has highest weight $\omega_1$ and dimension $n=r+1$.

\begin{proposition}\label{prop:wedge}
For the simple Lie algebra of type $A_r$ graded by an abelian group $G$ and for any $i=1,\ldots,r$, we have $\Br(\omega_i)=\Br(\omega_1)^i$ in the $(G/H_{\omega_1})$-graded Brauer group.
\end{proposition}

\begin{proof}
Denote $\bG=G/H_{\omega_1}$, $\cR=\End(V)$ and $\rho\colon\cL\to\cR$ the natural representation. Then the algebra $\wt{\cR}=\End(V^{\ot i})$ is $\bG$-graded by identification with $\cR^{\ot i}$, and $\rho^{\ot i}\colon U(\cL)\to\wt{\cR}$ is a homomorphism of $\bG$-graded algebras. We identify $\wedge^i V$ with the space of skew-symmetric tensors in $V^{\ot i}$. Let $\veps\in\wt{\cR}$ be the standard projection $V^{\ot i}\to\wedge^i V$, i.e.,
\[
\veps(v_1\ot\cdots\ot v_i)=\frac{1}{i!}\sum_{\pi\in S_i}(-1)^{\pi}v_{\pi(1)}\ot\cdots\ot v_{\pi(i)}\quad\mbox{for all}\;v_1,\ldots,v_i\in V.
\]
Then $\End(\wedge^i V)$ can be identified with $\veps\wt{\cR}\veps$. With this identification, we have $\rho^{\wedge i}(a)=\veps\rho^{\ot i}(a)\veps$ for all $a\in U(\cL)$.

Now let $K$ be the group of characters of $\bG$. If the $\bG$-grading on $\cR$ is associated to the action $\chi*x=u_\chi x u_\chi^{-1}$, $\chi\in K$, $x\in\cR$, then the grading on $\wt{\cR}$ is associated to the action $\chi*x=\tilde{u}_\chi x\tilde{u}_\chi^{-1}$, $\chi\in K$, $x\in\wt{\cR}$, where
\[
\tilde{u}_\chi(v_1\ot\cdots\ot v_i)=u_\chi(v_1)\ot\cdots\ot u_\chi(v_i)\quad\mbox{for all}\;v_1,\ldots,v_i\in V.
\]
Clearly, $\tilde{u}_\chi\veps\tilde{u}_\chi^{-1}=\veps$ for all $\chi\in K$, so $\veps$ is a homogeneous idempotent of $\wt{\cR}$. Therefore,
\[
\Br(\omega_i)=[\End(\wedge^i V)]=[\veps\wt{\cR}\veps]=[\wt{\cR}]=[\cR]^i=\Br(\omega_1)^i,
\]
where we have used Lemma \ref{lm:corner} (cf. the proof of Proposition \ref{prop:product}).
\end{proof}

\subsection{Inner gradings on $\Sl_{r+1}(\FF)$}\label{ss:A_I}

Gradings of Type I, i.e., such that the image of $\wh{G}$ in $\Aut(\cL)$ consists of inner automorphisms, are classified by the corresponding gradings on $\cR=\End(V)$, and the latter by the graded division algebra $\cD$ representing the class $[\cR]$ in the $G$-graded Brauer group and the multiset $\Xi$ in $G/T$ (determined up to a shift) representing the grading on a right vector space $W$ over $\cD$ such that $\cR=\End_{\cD}(W)$, where $T\subset G$ is the support of $\cD$. Explicitly, we can select a homogeneous $\cD$-basis $\{v_1,\ldots,v_k\}$ of $W$ and identify $V$ with $W\ot_\cD N=\wt{W}\ot N$ where $\wt{W}=\lspan{v_1,\ldots,v_k}$ (over $\FF$) and $N$ is the natural left (ungraded) module for the matrix algebra $\cD$, so $\dim N=\ell$ is the graded Schur index of $V$ and $n=k\ell$.

\begin{remark}
The $G$-grading on $V^\ell$ in Proposition \ref{prop:Schur_index}, where $\lambda=\omega_1$ and $H_\lambda$ is trivial, is obtained by identifying $V^\ell$ with $V\ot M=\wt{W}\ot\cD=W$ where $M$ is the natural right module for the matrix algebra $\cD$.
\end{remark}

The $G$-grading on $\cR$ then comes from the identification of $\cR$ with $\cC\ot\cD$ where $\cC=\End(\wt{W})$. The multiset $\Xi$ is $\{g_1T,\ldots,g_kT\}$ where $g_i=\deg v_i$. In \cite{BK10,EKmon}, the $\cD$-basis was chosen in such a way that $g_i=g_j$ whenever $g_iT=g_jT$; here it will be convenient not to impose this condition. Finally, the $G$-grading on $\cL$ comes from the identification of $\cL$ with the graded subspace $[\cR,\cR]$ of $\cR$. Two Type I gradings, $\Gamma$ and $\Gamma'$, belong to one $\inaut(\cL)$-orbit if and only if $T'=T$, $\beta'=\beta$ and $\Xi'=g\Xi$ for some $g\in G$ (see \cite{BK10} or \cite[Theorem 3.53]{EKmon}).

\begin{theorem}\label{th:A_inner}
Let $\cL$ be the simple Lie algebra of type $A_r$ over an algebraically closed field $\FF$ of characteristic $0$. Suppose $\cL$ is graded by an abelian group $G$ such that the image of $\wh{G}$ in $\Aut(\cL)$ consists of inner automorphisms (Type I grading). Then, for any dominant integral weight $\lambda=\sum_{i=1}^r m_i\omega_i$, we have $H_\lambda=\{e\}$ and $\Br(\lambda)=\hat{\beta}^{\sum_{i=1}^r im_i}$, where $\hat{\beta}\colon\wh{G}\times\wh{G}\to\FF^\times$ is the commutation factor associated to the parameters $(T,\beta)$ of the grading on $\cL$.
\end{theorem}

\begin{proof}
Since $H_{\omega_i}$ is trivial for all $i$, we can apply Proposition \ref{prop:product}:
\[
\Br(\lambda)=\prod_{i=1}^r\Br(\omega_i)^{m_i}.
\]
But by Proposition \ref{prop:wedge}, we have $\Br(\omega_i)=\Br(\omega_1)^i=\hat{\beta}^i$. The result follows.
\end{proof}

\begin{corollary}
The simple $\cL$-module $V_\lambda$ admits a $G$-grading making it a graded $\cL$-module if and only if the number $\sum_{i=1}^r im_i$ is divisible by the exponent of the group $T$.
\end{corollary}

\begin{example}
Consider $\cL=\Sl_2(\FF)$. Then either $T=\{e\}$ or $T=\langle a\rangle\times\langle b\rangle\cong\ZZ_2^2$. In the first case, any simple $\cL$-module has trivial Brauer invariant and hence admits a $G$-grading, as expected because the grading on $\cL$ is a coarsening of Cartan grading. In the second case (so-called Pauli grading on $\cL$), $T$ has a unique nondegenerate alternating bicharacter $\beta$ (defined on the generators by $\beta(a,a)=\beta(b,b)=1$ and $\beta(a,b)=\beta(b,a)=-1$), for which $\beta^2=1$. Hence $\Br(m\omega_1)$ is trivial for even $m$ and equals $\hat{\beta}$ for odd $m$. In particular, the simple $\cL$-modules of even highest weight admit a $G$-grading but those of odd highest weight do not: they have graded Schur index $2$.
\end{example}

\subsection{Outer gradings on $\Sl_{r+1}(\FF)$}\label{ss:A_II}

If $r>1$ then there exist gradings of Type II, i.e., such that the image of $\wh{G}$ in $\Aut(\cL)$ contains an outer automorphism. Every such grading has a distinguished element $h\in G$ of order $2$, which is characterized by the property that the corresponding $\bG$-grading is of Type I, where $\bG=G/\langle h\rangle$. (Using our present notation, we have $H_{\omega_1}=\langle h\rangle$.) Gradings of Type II with a fixed distinguished element $h$ are classified by the corresponding $\bG$-graded algebras  $\cR$ equipped with an anti-automorphism $\vphi$ representing (the negative of) the action of a fixed $\chi\in\wh{G}$ with $\chi(h)=-1$. Namely, the Type II grading on $\cL$ corresponding to $(\cR,\vphi)$ comes from the identification of $\cL$ with the graded subspace $[\cR,\cR]$ of the Lie algebra $\brac{\cR}$  equipped with the $G$-grading given by \eqref{eq:phi_refinement}.

The existence of $\vphi$ forces $\bT\subset\bG$ to be an elementary $2$-group, where $\bT$ is the support of the graded division algebra $\cD$ representing $[\cR]$ in the $\bG$-graded Brauer group. Moreover, for any $x\in\cR$, $\vphi(x)$ is the adjoint of $x$ with respect to a suitable nondegenerate $\FF$-bilinear form $B\colon W\times W\to\cD$ which is $\cD$-sesquilinear relative to an involution of the graded algebra $\cD$ (see \cite{E09d} or \cite{EKmon}). To write $\vphi$ explicitly, we can fix a ``standard realization'' of $\cD$ as a matrix algebra in the following way. To simplify notation, we temporarily omit bars and write $T$ instead of $\bT$. Select a symplectic basis $\{a_j\}\cup\{b_j\}$ of $T$ (as a vector space over the field of two elements) with respect to the nondegenerate alternating bicharacter $\beta\colon T\times T\to\FF^\times$, i.e., $\beta(a_j,b_j)=-1$ and all other values of $\beta$ on basis elements are equal to $1$. Then we can take $\cD=M_2(\FF)\ot\cdots\ot M_2(\FF)$ with the following $T$-grading:
\[
X_{a_j}=1\ot\cdots\ot\matr{-1&0\\0&1}\ot\cdots\ot 1\quad\mbox{and}\quad
X_{b_j}=1\ot\cdots\ot\matr{0&1\\1&0}\ot\cdots\ot 1,
\]
where the indicated matrix appears in the $j$-th factor, and
\[
X_{(\prod a_j^{\xi_j})(\prod b_j^{\eta_j})}=
\big(\prod X_{a_j}^{\xi_j}\big)\big(\prod X_{b_j}^{\eta_j}\big).
\]
Note that we have ${}^tX_{a_j}=X_{a_j}^{-1}=X_{a_j}$ and ${}^tX_{b_j}=X_{b_j}^{-1}=X_{b_j}$ for all $j$, and hence
\begin{equation}\label{df:quadratic_form_beta}
{}^tX_s=X_s^{-1}=\beta(s)X_s\quad\text{for all $s\in T$},
\end{equation}
where $\beta\colon T\to\{\pm 1\}$ is a quadratic form on $T$ (as a vector space over the field of two elements) whose polar bilinear form is $\beta(\cdot,\cdot)$, i.e., $\beta(s,t)=\beta(st)\beta(s)\beta(t)$ for all $s,t\in T$.

\begin{remark}\label{rem:matrix_realization}
This matrix realization of $\cD$ can be presented in a concise way if we write $T=A\times B$, where $A=\langle a_j\rangle$ and $B=\langle b_j\rangle$, and take for $N$ the vector space $\FF B=\lspan{e_b\;|\;b\in B}$ with the following action of $\cD$: $X_{(a,b)}e_{b'}=\beta(a,bb')e_{bb'}$ for all $a\in A$ and $b,b'\in B$.
\end{remark}

\begin{lemma}\label{lm:det}
If $|T|\ne 4$ then $\det(X_t)=1$ for all $t\in T$. If $|T|=4$ then $\det(X_t)=\beta(c,t)$ for all $t\in T$, where $c=ab$.\qed
\end{lemma}

Fix a standard matrix realization of $\cD$. Adjusting the sesquilinear form $B$, we may assume that the involution of $\cD$ is the matrix transpose. Selecting a suitable homogeneous $\cD$-basis $\{v_1,\ldots,v_k\}$ in $W$, $\deg v_i=\bg_i$, we may assume that
\begin{equation}\label{eq:relation_bg0}
\bg_1^2\bt_1=\ldots=\bg_q^2\bt_q=\bg_{q+1}\bg_{q+2}=\ldots=\bg_{q+2s-1}\bg_{q+2s}=\bg_0^{-1},
\end{equation}
where $q+2s=k$, $\bg_0\in\bG$, $\bt_i\in\bT$, and the sesquilinear form $B$ is represented by the following block-diagonal matrix:
\begin{equation}\label{eq:Phi}
\Phi=\diag\left(X_{t_1},\ldots,X_{t_q},\matr{0&I\\\mu_1 I&0},\ldots,\matr{0&I\\\mu_s I&0}\right),
\end{equation}
where $\mu_i\in\FF^\times$ and $I=X_e\in\cD$ (see \cite{E09d} or \cite[Theorem 3.31]{EKmon}).
The element $\bg_0$ has the meaning of the degree of $B$ as a linear map $W\ot W\to\cD$. The scalars $\mu_i$ can be expressed in terms of a single scalar $\mu_0$ satisfying $\mu_0^2=\chi^2(\bg_0^{-1})$, namely, $\mu_i=\mu_0\chi^2(\bg_{q+2i-1}^{-1})$. The gradings of Type II with distinguished element $h$ are classified by the graded division algebra $\cD$, the multiset $\Xi=\{\bg_1\bT,\ldots,\bg_k\bT\}$, and the elements $\mu_0\in\FF^\times$ and $\bg_0\in\bG$ (see \cite{BK10} or \cite[Theorem 3.53]{EKmon}). Identifying $\cR$ with $M_k(\cD)$ through the $\cD$-basis $\{v_1,\ldots,v_k\}$ and using the fixed matrix realization of $\cD$, we can write the anti-automorphism $\vphi$ in matrix form as follows:
\[
\vphi(X)=\Phi^{-1}({}^t X)\Phi\quad\mbox{for all}\; X\in M_n(\FF),
\]
where ${}^t X$ is the transpose of $X$. Note that $\vphi(X)$ is the adjoint of $X$ with respect to the nondegenerate bilinear form $(v',v'')_\Phi={}^tv'\Phi v''$ on $V$, which can also be expressed, using the identification $V=W\ot_\cD N=\wt{W}\ot N$, as follows:
\begin{equation}\label{eq:form_Phi}
(w'\ot x,w''\ot y)_\Phi={}^t xB(w',w'')y\quad\mbox{for all}\;w',w''\in W,\,x,y\in N,
\end{equation}
where the expression in the right-hand side is matrix product (using the fixed matrix realization of $\cD$).

Recall that the action of a character $\psi$ of $\bG$ on the $\bG$-graded algebra $\cR$ is given by $\psi*x=u_\psi xu_\psi^{-1}$ ($x\in\cR$) where
\begin{equation}\label{eq:u_psi}
u_\psi=\diag(\psi(\bg_1),\ldots,\psi(\bg_k))\ot X_{\bt}
\end{equation}
and $\bt\in\bT$ is determined by the condition $\beta(\bt,\bar{s})=\psi(\bar{s})$ for all $\bar{s}\in\bT$.

\begin{lemma}\label{lm:action_on_form}
For any $\psi\in\wh{G}$ with $\psi(h)=1$, we have $(u_\psi v',u_\psi v'')_\Phi=\psi(\bg_0^{-1})(v',v'')_\Phi$ for all $v',v''\in V$.
\end{lemma}

\begin{proof}
It is sufficient to verify ${}^tu_\psi\Phi u_\psi=\psi(\bg_0^{-1})\Phi$, which follows from orthogonality of $X_{\bt}$ and
the equations
\[
\psi(\bg_i)^2 X_{\bt}^{-1}X_{\bt_i}X_{\bt}=\psi(\bg_i)^2 \beta(\bt,\bt_i)X_{\bt_i}=\psi(\bg_i)^2 \psi(\bt_i)X_{\bt_i}=\psi(\bg_0^{-1})X_{\bt_i},
\]
for each $i=1,\ldots,q$, as well as the equations $\psi(\bg_{q+2j-1})\psi(\bg_{q+2j})=\psi(\bg_0^{-1})$, for each $j=1,\ldots,s$.
\end{proof}

We will need the following notation to state our main result on Type II gradings. Let $\ell$ be the $\bG$-graded Schur index of $V$, so $\ell$ is a power of $2$, $|\bT|=\ell^2$ and $\dim V=n=k\ell$. If $n$ is even, then set
\begin{equation}\label{df:h_prime}
\bar{h}'=\left\{\begin{array}{ll}
\bg_0^{n/2}(\bg_1\cdots\bg_k)^\ell & \mbox{if}\;\ell\ne 2,\\
(\bar{c}\bg_0)^{n/2}(\bg_1\cdots\bg_k)^\ell & \mbox{if}\;\ell=2,
\end{array}\right.
\end{equation}
where, for $\ell=2$, $\bar{c}$ is the element of $\bT$ as in Lemma \ref{lm:det} ($\bar{c}=\bar{a}\bar{b}$ where $\{\bar{a},\bar{b}\}$ is the selected basis of $\bT$). Note that relations \eqref{eq:relation_bg0} imply that $\bar{h}'$ has order at most $2$ and depends only on $\bg_1,\ldots,\bg_q$.

\begin{theorem}\label{th:A_outer}
Let $\cL$ be the simple Lie algebra of type $A_r$ ($r>1$) over an algebraically closed field $\FF$ of characteristic $0$. Suppose $\cL$ is graded by an abelian group $G$ such that the image of $\wh{G}$ in $\Aut(\cL)$ contains outer automorphisms and hence $H_{\omega_1}=\langle h\rangle$ for some $h\in G$ of order $2$ (Type II grading with distinguished element $h$). Let $K=K_{\omega_1}=\langle h\rangle^\perp$ and fix $\chi\in\wh{G}\setminus K$ (i.e., $\chi(h)=-1$). Then, for a dominant integral weight $\lambda=\sum_{i=1}^r m_i\omega_i$, we have the following possibilities:
\begin{enumerate}
\item[1)]
If $m_i\ne m_{r+1-i}$ for some $i$, then $H_\lambda=\langle h\rangle$, $K_\lambda=K$, and
$\Br(\lambda)=\hat{\beta}^{\sum_{i=1}^r im_i}=\hat{\beta}^{\sum_{i=1}^{\lfloor (r+1)/2\rfloor} m_{2i-1}}$, where $\hat{\beta}\colon K\times K\to\FF^\times$ is the commutation factor associated to the parameters $(\bT,\beta)$ of the grading on $\cL$.
\item[2a)]
If $r$ is even and $m_i=m_{r+1-i}$ for all $i$, then $H_\lambda=\{e\}$ and $\Br(\lambda)=1$.
\item[2b)]
If $r$ is odd and $m_i=m_{r+1-i}$ for all $i$, then $H_\lambda=\{e\}$ and $\Br(\lambda)=\hat{\gamma}^{m_{(r+1)/2}}$, where $\hat{\gamma}\colon\wh{G}\times\wh{G}\to\FF^\times$ is the extension of the alternating bicharacter $\hat{\beta}^{(r+1)/2}\colon K\times K\to\FF^\times$ given by $\hat{\gamma}(\chi,\psi)=\psi(h')$ for all $\psi\in\wh{G}$, where $\hat{\beta}$ is the commutation factor associated to the parameters $(\bT,\beta)$ of the grading on $\cL$ and $h'\in G$ is the unique element satisfying $\chi(h')=1$ in the coset $\bar{h}'$ defined by \eqref{df:h_prime}. The order of $h'$ is at most $2$.
\end{enumerate}
\end{theorem}

\begin{proof}
Since $\chi$ acts by an outer automorphism, we have $V_{r+1-i}\cong V_i^{\chi}$. The $\wh{G}$-orbits in the set of fundamental weights are the pairs $\{\omega_i,\omega_{r+1-i}\}$, $i=1,\ldots,\lfloor r/2\rfloor$, and, if $r$ is odd, also the singleton $\{\omega_{(r+1)/2}\}$.

1) In this case, we actually deal with a Type I grading by $\bG=G/\langle h\rangle$, so Theorem \ref{th:A_inner} applies.

2a) Here $\lambda=\sum_{i=1}^{r/2}m_i(\omega_i+\omega_{r+1-i})$. Since $H_{\omega_i+\omega_{r+1-i}}$ is trivial for all $i$, we can apply Proposition \ref{prop:product}:
\[
\Br(\lambda)=\prod_{i=1}^{r/2}\Br(\omega_i+\omega_{r+1-i})^{m_i}.
\]
Since $V_{\omega_i}$ and $V_{\omega_{r+1-i}}$ are dual to each other, Proposition \ref{prop:product_with_dual} gives the result.

2b) Now $\lambda=m_{(r+1)/2}\omega_{(r+1)/2}+\sum_{i=1}^{(r-1)/2}m_i(\omega_i+\omega_{r+1-i})$ and hence
\[
\Br(\lambda)=\Br(\omega_{(r+1)/2})^{m_{(r+1)/2}}\prod_{i=1}^{(r-1)/2}\Br(\omega_i+\omega_{r+1-i})^{m_i}=\Br(\omega_{(r+1)/2})^{m_{(r+1)/2}}
\]
by Propositions \ref{prop:product} and \ref{prop:product_with_dual}. It remains to consider the module $V_{\omega_p}=\wedge^p V$ for $p=\frac{r+1}{2}$ and compute the associated commutation factor $\hat{\gamma}\colon\wh{G}\times\wh{G}\to\FF^\times$. The restriction of $\hat{\gamma}$ to $K\times K$ is equal to $\hat{\beta}^p$ according to Proposition \ref{prop:wedge}. Since $\hat{\gamma}(\chi,\chi)=1$, it suffices to show that $\hat{\gamma}(\chi,\psi)=\psi(\bar{h}')$ for all $\psi\in K$. By definition of the commutation factor, we have
\begin{equation}\label{df:gamma_hat}
\tilde{u}_\chi \tilde{u}_\psi=\hat{\gamma}(\chi,\psi)\tilde{u}_\psi \tilde{u}_\chi,
\end{equation}
where $\tilde{u}_\chi\colon(\wedge^p V)^{\chi^{-1}}\to\wedge^p V$ and $\tilde{u}_\psi\colon(\wedge^p V)^{\psi^{-1}}\to\wedge^p V$ are isomorphisms of $\cL$-modules. We can take for $\tilde{u}_\psi$ the mapping $x_1\wedge\cdots\wedge x_p\mapsto u_\psi x_1\wedge\cdots\wedge u_\psi x_p$ ($x_i\in V$) where $u_\psi$ is given by \eqref{eq:u_psi}. We will now describe $\tilde{u}_\chi$.

Let $u'\colon V^{\chi^{-1}}\to V^*$ be an isomorphism of $\cL$-modules. Then, for any $a\in\cR$, $\vphi(a)$ is the adjoint of $a$ with respect to the bilinear form $(x,y)=\langle u'x,y\rangle$ on $V$ --- see the proof of Proposition \ref{prop:product_with_dual} (with $\lambda_1=\omega_1$). Multiplying $u'$ by a scalar if necessary, we may assume that $(x,y)_\Phi=\langle u'x,y\rangle$ for all $x,y\in V$. Let $\tilde{u}$ be the induced isomorphism $(\wedge^p V)^{\chi^{-1}}\to\wedge^p(V^*)$, which is given by $\tilde{u}(x_1\wedge\cdots\wedge x_p)=u'x_1\wedge\cdots\wedge u'x_p$ for all $x_i\in V$. The $\cL$-module $\wedge^p(V^*)$ can be identified with $(\wedge^p V)^*$ through the canonical pairing
\[
\langle f_1\wedge\cdots\wedge f_q,y_1\wedge\cdots\wedge y_q\rangle
=\sum_{\pi\in S_p}(-1)^\pi\langle f_1,y_{\pi(1)}\rangle\cdots\langle f_p,y_{\pi(p)}\rangle
\;\mbox{for all}\; y_i\in V,\, f_i\in V^*.
\]
Hence $\langle\tilde{u}x,y\rangle=(x,y)_\Phi$ for all $x,y\in\wedge^p V$, where the bilinear form $(\cdot,\cdot)_\Phi$ is defined on $\wedge^p V$ as follows:
\[
(x_1\wedge\cdots\wedge x_q,y_1\wedge\cdots\wedge y_q)_\Phi
=\sum_{\pi\in S_p}(-1)^\pi(x_1,y_{\pi(1)})_\Phi\cdots(x_p,y_{\pi(p)})_\Phi
\;\mbox{for all}\; x_i,y_i\in V.
\]
Note that, by Lemma \ref{lm:action_on_form}, we have
\begin{equation}\label{eq:action_on_form}
(\tilde{u}_\psi x,\tilde{u}_\psi y)_\Phi=\psi(\bg_0^{-p})(x,y)_\Phi\quad\mbox{for all}\;x,y\in\wedge^p V.
\end{equation}
Since $2p=\dim V$, the $\cL$-module $\wedge^p V$ is isomorphic to its dual $(\wedge^p V)^*$ via the bilinear form on $\wedge^p V$ given by $(x,y)=x\wedge y$. Namely, an isomorphism
$\theta\colon\wedge^p V\to(\wedge^p V)^*$ is defined by the rule $\langle\theta(x),y\rangle=(x,y)$ for all $x,y\in\wedge^p V$. Hence $\theta^{-1}\tilde{u}$ is an isomorphism $(\wedge^p V)^{\chi^{-1}}\to\wedge^p V$, so we can take $\tilde{u}_\chi=\theta^{-1}\tilde{u}$. By construction, we have
\[
(\tilde{u}_\chi x,y)=\langle\tilde{u}x,y\rangle=(x,y)_\Phi\quad\mbox{for all}\;x,y\in\wedge^p V.
\]

Finally, we can calculate $\hat{\gamma}(\chi,\psi)$ using equation \eqref{df:gamma_hat}. For any $x,y\in\wedge^p V$, on the one hand, we obtain
\[
(\tilde{u}_\chi\tilde{u}_\psi x,y)=(\tilde{u}_\psi x,y)_\Phi=\psi(\bg_0^{-p})(x,\tilde{u}_\psi^{-1}y)_\Phi,
\]
where we have used \eqref{eq:action_on_form}. On the other hand, we obtain
\[
(\tilde{u}_\psi\tilde{u}_\chi x,y)=\det(u_\psi)(\tilde{u}_\chi x,\tilde{u}_\psi^{-1}y)=\det(u_\psi)(x,\tilde{u}_\psi^{-1}y)_\Phi.
\]
It follows that
\[
\hat{\gamma}(\chi,\psi)^{-1}=\psi(\bg_0^p)\det(u_\psi).
\]
Looking at \eqref{eq:u_psi}, we see that $\det(u_\psi)=\big(\prod_{i=1}^k\psi(\bg_i)^\ell\big)\det(X_{\bt})^k$. Taking into account Lemma \ref{lm:det} and recalling that $\beta(\bar{c},\bt)=\psi(\bar{c})$ by definition of $u_\psi$, we obtain
\[
\hat{\gamma}(\chi,\psi)^{-1}=\left\{\begin{array}{ll}
\psi(\bg_0^p)\psi\big(\prod_{i=1}^k\bg_i^\ell\big) & \mbox{if}\;\ell\ne 2,\\
\psi(\bg_0^p)\psi\big(\prod_{i=1}^k\bg_i^\ell\big)\psi(\bar{c})^k & \mbox{if}\;\ell=2,
\end{array}\right.
\]
so $\hat{\gamma}(\chi,\psi)=\psi(\bar{h}')^{-1}=\psi(\bar{h}')$ since $\bar{h}'$ has order at most $2$. Substituting $\psi=\chi^2$, we get $1=\hat{\gamma}(\chi,\chi^2)=\chi^2(\bar{h}')$. Let $h'$ be one of the two elements in the coset $\bar{h}'$. Since $\chi(h')^2=\chi^2(\bar{h}')=1$, we see that $(h')^2\ne h$ and hence $(h')^2=e$. Replacing $h'$ with $h'h$ if necessary, we obtain $\chi(h')=1$. Then $\hat{\gamma}(\chi,\psi)=\psi(h')$ for all $\psi\in\wh{G}$.
\end{proof}

\begin{remark}
If $r\equiv 3\pmod{4}$ then $\hat{\beta}^{(r+1)/2}=1$ and the support $T$ of the graded division algebra representing $\Br(\omega_{(r+1)/2})$ is $\{e\}$ if $h'=e$ and $\langle h, h'\rangle\cong\ZZ_2^2$ if $h'\ne e$.
\end{remark}

\begin{proof}
The property $\hat{\gamma}(\chi,\psi)=\psi(h')$ implies that $\rad\hat{\gamma}=\wh{G}$ if $h'=e$ and $\rad\hat{\gamma}=K\cap\langle h'\rangle^\perp$ if $h'\ne e$. Since $T=(\rad\hat{\gamma})^\perp$, we get $T=\{e\}$ if $h'=e$ and $T=\langle h,h'\rangle$ if $h'\ne e$.
\end{proof}

\begin{remark}
If $r\equiv 1\pmod{4}$ then $\hat{\beta}^{(r+1)/2}=\hat{\beta}$ and the support $T$ of the graded division algebra representing $\Br(\omega_{(r+1)/2})$ is $\bT$ if $\bar{h}'\in\bT$ and $\bT\times\langle h, h'\rangle\cong\bT\times\ZZ_2^2$ (orthogonal sum relative to $\gamma$) if $\bar{h}'\notin\bT$, where $\bT$ is regarded as a subgroup of $G$ by identifying $\bt\in\bT$ with the unique element $t$ in the coset $\bt$ satisfying $\chi(t)=\beta(\bar{h}',\bt)$ in the case $\bar{h}'\in\bT$ and $\chi(t)=1$ in the case $\bar{h}'\notin\bT$.
\end{remark}

\begin{proof}
The property $\hat{\gamma}(\chi,\psi)=\psi(h')$ implies that $\rad\hat{\gamma}=\rad\hat{\beta}\cap\langle h'\rangle^\perp$ if $\rad\hat{\beta}\not\subset\langle h'\rangle^\perp$. The condition $\rad\hat{\beta}\subset\langle h'\rangle^\perp$ is equivalent to $\bar{h}'\in\bT$. If this is the case, then $\rad\hat{\gamma}=\langle\rad\hat{\beta},\tilde\chi\rangle$  where $\tilde\chi=\chi\psi_0$ and $\psi_0$ is any element of $K$ satisfying $\psi_0|_{\bT}=\beta(\bar{h}',\cdot)$. Let $H$ be the pre-image of $\bT$ under the quotient map $G\to\bG$, so $H^\perp=\rad\hat{\beta}$.

If $\bar{h}'\in\bT$ then $T=(\rad\hat{\gamma})^\perp=H\cap\langle\tilde{\chi}\rangle^\perp$ is a subgroup of index $2$ in $H$ (since $\tilde{\chi}\notin\rad\hat{\beta}$ and $\tilde{\chi}^2\in\rad\hat{\beta}$), so $T$ is precisely the isomorphic copy of $\bT$ in $G$ obtained as indicated.

If $\bar{h}'\notin\bT$ then $T=\langle H,h'\rangle$. The definition of $\hat{\gamma}(\cdot,\cdot)$ and the fact that $\hat{\gamma}(\chi,\psi)=\psi(h')$ for all $\psi\in\wh{G}$ imply that $\gamma(t,h')=\chi(t)$ for all $t\in T$. In particular, $\gamma(h,h')=-1$ and hence the restriction of $\gamma(\cdot,\cdot)$ to $\langle h,h'\rangle$ is nondegenerate. Also, $\psi(h)=\hat{\gamma}(\chi',\psi)$ for all $\psi\in\wh{G}$, where $\chi'\in\rad\hat{\beta}$. It follows that $T=\bT\times\langle h,h'\rangle$, where $\bT$ is a subgroup of $G$ obtained as indicated.
\end{proof}

\begin{corollary}
The simple $\cL$-module $V_\lambda$ admits a $G$-grading making it a graded $\cL$-module if and only if \textup{1)} for all $i$, $m_i=m_{r+1-i}$ and \textup{2)} either $r$ is even or $r$ is odd and one of the following conditions is satisfied: \textup{(i)} $m_{(r+1)/2}$ is even or \textup{(ii)} $r\equiv 3\pmod{4}$ and $\bar{h}'=\bar{e}$ in $\bG$ or \textup{(iii)} $r\equiv 1\pmod{4}$, $\bT=\{\bar{e}\}$ and $\bar{h}'=\bar{e}$ in $\bG$.
\end{corollary}

\section{Series B}\label{s:B}

Consider the simple Lie algebra $\cL=\So_{2r+1}(\FF)$ of type $B_{r}$, $r\ge 2$, where $\FF$ is an algebraically closed field of characteristic $0$. In this case all automorphisms of $\cL$ are inner. For a given $G$-grading on $\cL$, we are going to find the Brauer invariants of all simple $\cL$-modules.

In \cite{BK10}, the $G$-gradings on $\cL$ are classified in terms of the corresponding $G$-graded algebras $\cR=\End(V)$ equipped with an orthogonal involution $\vphi$ where $V$ is the natural module of $\cL$, which has highest weight $\omega_1$ and dimension $n=2r+1$. Namely, the grading on $\cL$ corresponding to $(\cR,\vphi)$ comes from the identification of $\cL$ with the graded subspace $\sks(\cR,\vphi)=\{x\in\cR\;|\;\vphi(x)=-x\}$.

The existence of the involution $\vphi$ forces $T\subset G$ to be an elementary $2$-group, where $T$ is the support of the graded division algebra $\cD$ representing $[\cR]$ in the $G$-graded Brauer group. But $|T|$ is a divisor of the odd number $n^2$, hence $T=\{e\}$, i.e., the Brauer invariant of the natural module is trivial. In other words, the $G$-grading on $\cR=\End(V)$ is induced by some $G$-grading $V=\bigoplus_{g\in G}V_g$, which is determined up to a shift. Since, for $i=1,\ldots,r-1$, the simple $\cL$-module of highest weight $\omega_i$ can be realized as $\wedge^i V$, which has a natural $G$-grading induced from $V$, we see that the Brauer invariants of all fundamental weights, except possibly $\omega_r$, are trivial. The simple $\cL$-module $V_{\omega_r}$ is called the {\em spin module} and can be constructed in the following way, using the Clifford algebra $\Cl(V,Q)$ where $Q$ is the quadratic form associated to the orthogonal involution $\vphi$.

For any $x\in\cR$, $\vphi(x)$ is the adjoint of $x$ with respect to a nondegenerate symmetric $\FF$-bilinear form $B\colon V\times V\to\FF$. Selecting a suitable homogeneous basis $\{v_1,\ldots,v_n\}$ in $V$, $\deg v_i=g_i$, we may assume that
\begin{equation}\label{eq:relation_g0_trivT}
g_1^2=\ldots=g_q^2=g_{q+1}g_{q+2}=\ldots=g_{q+2s-1}g_{q+2s}=g_0^{-1},
\end{equation}
where $q+2s=n$, $g_0\in G$, and the bilinear form $B$ is represented by the following block-diagonal matrix:
\begin{equation}\label{eq:PhiB}
\Phi=\diag\left(1,\ldots,1,\matr{0&1\\ 1&0},\ldots,\matr{0&1\\ 1&0}\right),
\end{equation}
where the number of $1$'s is $q$ (see \cite{E09d} or \cite[Theorem 3.42]{EKmon}). The element $g_0$ has the meaning of the degree of $B$ as a linear map $V\ot V\to\FF$. Since $q$ is odd, we have $q\ge 1$ and hence $g_0$ is a square in $G$, so we may, and will, shift the $G$-grading on $V$ so that $g_0$ becomes $e$. The $G$-gradings on $\cL$ are classified by the multiset $\Xi=\{g_1,\ldots,g_n\}$ (see \cite{BK10} or \cite[Theorem 3.65]{EKmon}).

\begin{remark}
Since $\FF$ is algebraically closed, we may take, at the expense of adding blocks $\matr{0&1\\ 1&0}$, the elements $g_1,\ldots,g_q$ to be distinct.
\end{remark}

The quadratic form $Q$ is given by $Q(v)=\frac{1}{2}B(v,v)$ so that $B$ is the polar form of $Q$, i.e., $B(u,v)=Q(u+v)-Q(u)-Q(v)$. In particular $Q(v_i)=\frac{1}{2}$ for $i=1,\ldots,q$ and $Q(v_i)=0$ for $i=q+1,\ldots,q+2s$.

Recall that the Clifford algebra $\Cl(V)=\Cl(V,Q)$ of a quadratic space $(V,Q)$ is the quotient of the tensor algebra $T(V)$ by the ideal generated by the elements $v\otimes v-Q(v)1$, for all $v\in V$. We will denote by $x\cdot y$ the product of elements in $\Cl(V)$. The algebra $\Cl(V)$ is naturally graded by $\ZZ_2$: $\Cl(V)=\Cl\subo(V)\oplus\Cl\subuno(V)$, with $V$ contained in the odd part.

The Lie algebra $\frso(V,Q)$ imbeds in $\Cl(V,Q)^{(-)}$ (actually, in its even part) as the subspace $[V,V]^\cdot=\lspan{u\cdot v-v\cdot u\;|\; u,v\in V}$. Indeed, for $u,v,w\in V$,
\begin{equation}\label{eq:[uv]sigmauv}
\begin{split}
\ad_{[u,v]^\cdot}(w)&=[[u,v]^\cdot,w]^\cdot\\
    &=u\cdot v\cdot w-v\cdot u\cdot w-w\cdot u\cdot v+w\cdot v\cdot u\\
    &=\bigl(B(v,w)u-u\cdot w\cdot v\bigr)-\bigl(B(u,w)v-v\cdot w\cdot u\bigr)\\
    &\qquad -\bigl(B(u,w)v-u\cdot w\cdot v\bigr)+\bigl(B(v,w)u-v\cdot w\cdot u\bigr)\\
    &=-2\bigl(B(u,w)v-B(v,w)u\bigr),
\end{split}
\end{equation}
and the linear maps
\begin{equation}\label{eq:sigmauv}
\sigma_{u,v}:w\mapsto B(u,w)v-B(v,w)u
\end{equation}
span $\frso(V,Q)$. Hence we have an imbedding $\rho\colon\frso(V,Q)\to\Cl\subo(V,Q)$ defined by the property $\rho(\sigma_{u,v})=-\frac{1}{2}[u,v]^\cdot$.

In the case at hand, $\dim V$ is odd, so the even Clifford algebra $\Cl\subo(V)$ is simple, and its unique (up to isomorphism) irreducible module is the spin module of $\cL$.

The $G$-grading on $V$ induces a $G$-grading on $T(V)$, and the elements $v\otimes v-Q(v)1$ are homogeneous because $B$ is a homogeneous map $V\ot V\to\FF$ of degree $e$. Hence $\Cl(V)$ inherits a $G$-grading from $T(V)$, and $V$ imbeds in $\Cl(V)$ as a graded subspace. The $G$-grading and the $\ZZ_2$-grading on $\Cl(V)$ are compatible in the sense that the homogeneous components of one are graded subspaces with respect to the other. Since the $G$-gradings on $\cL$ and on $\Cl\subo(V)$ are both induced from $V$, the imbedding $\rho\colon\cL\to\brac{\Cl\subo(V)}$ is a homomorphism of $G$-graded algebras.

Recall that, for $\chi\in \wh{G}$, the action of $\chi$ on $V$ is given by the matrix:
\[
u_\chi=\diag(\chi(g_1),\ldots,\chi(g_{q+2s})),
\]
and the action of $\chi$ on $\cR$ and $\cL$ is the conjugation by $u_\chi$. Let $\lambda_i=\chi(g_i)$. Note that equation \eqref{eq:relation_g0_trivT}, with $g_0=e$, implies that $\lambda_i^2=1$ for $i=1,\ldots,q$ and $\lambda_{q+2j-1}\lambda_{q+2j}=1$ for $j=1,\ldots,s$. Thus $u_\chi$ is in the orthogonal group $\Ort(V,Q)$. Note that the square of the element $g_1\cdots g_q$ is $e$, so we may  shift the grading on $V$ by this element and still have $g_0=e$. Therefore, we may assume without loss of generality that $g_1\cdots g_q=e$. Then we have $\lambda_1\cdots\lambda_q=1$ and hence $\det(u_\chi)=\lambda_1\cdots\lambda_{q+2s}=1$, i.e., $u_\chi$ is actually in the special orthogonal group $\SO(V,Q)$.

For any $\chi\in \wh{G}$, the orthogonal transformation $u_\chi$ induces an automorphism $\Cl(u_\chi)$ of $\Cl(V,Q)$ whose restriction to $V$ is $u_\chi$. This automorphism gives the action of $\chi$ on the $G$-graded algebra $\Cl(V,Q)$ and on its graded subalgebra $\Cl\subo(V,Q)$.

Since $u_\chi\in\SO(V,Q)$, there is an element $s_\chi\in \Spin(V,Q)$, unique up to sign, such that $u_\chi(v)=s_\chi\cdot v\cdot s_\chi^{-1}$ for all $v\in V$. Recall that
\[
\Spin(V,Q)=\{x\in \Cl\subo(V,Q)\,|\,x\cdot V\cdot x^{-1}=V\ \text{and}\ x\cdot\tau(x)=1\},
\]
where $\tau$ is the \emph{canonical involution} of $\Cl(V,Q)$, i.e., the unique involution whose restriction to $V$ is the identity.

Given isotropic elements $u,v\in V$ with $B(u,v)=1$, and a nonzero scalar $\alpha\in\FF$, the element
\[
s=(\alpha u+\alpha^{-1}v)\cdot (u+v)=\alpha u\cdot v+\alpha^{-1}v\cdot u\in\Spin(V,Q)
\]
commutes with all the vectors in $V$ orthogonal to $u$ and $v$, and satisfies
\[
\begin{split}
s\cdot u&=(\alpha u\cdot v+\alpha^{-1}v\cdot u)\cdot u=\alpha u\cdot v\cdot u=\alpha u,\\
u\cdot s&=u\cdot (\alpha u\cdot v+\alpha^{-1}v\cdot u)=\alpha^{-1} u\cdot v\cdot u=\alpha^{-1}u,
\end{split}
\]
so $s\cdot u\cdot s^{-1}=\alpha^2 u$ and, in the same vein, $s\cdot v\cdot s^{-1}=\alpha^{-2}v$.

Hence, if we choose square roots $\lambda_{q+i}^{1/2}$ for $i=1,\ldots,2s$ in such a way that $\lambda_{q+2j-1}^{1/2}\lambda_{q+2j}^{1/2}=1$ for all $j=1,\ldots,s$, then we may take
\begin{equation}\label{eq:B_s_chi}
s_\chi=
\Bigl(\prod_{\substack{1\leq i\leq q\\
\lambda_i=-1}}2^{1/2}v_i\Bigr)\cdot\Bigl(\prod_{1\le j\le s}\bigl(\lambda_{q+2j-1}^{1/2}v_{q+2j-1}\cdot v_{q+2j}+\lambda_{q+2j}^{1/2}v_{q+2j}\cdot v_{q+2j-1}\bigr)\Bigr).
\end{equation}
Note that there is an even number of indices $i$ with $\lambda_i=-1$ ($1\leq i\leq q$).

But for $u,v$ as above and $0\ne \alpha,\beta\in \FF$, we have
\[
\begin{split}
(\alpha u\cdot v+\alpha^{-1}v\cdot u)\cdot (\beta u\cdot v+\beta^{-1}v\cdot u)&
=\alpha\beta u\cdot v\cdot u\cdot v+\alpha^{-1}\beta^{-1}v\cdot u\cdot v\cdot u\\
& =\alpha\beta u\cdot v+(\alpha\beta)^{-1}v\cdot u.
\end{split}
\]
Also, if $u_1,\ldots,u_m$ are orthogonal elements and $I$, $J$ are two subsets in $\{1,\ldots,m\}$ of even size, then $u_I\cdot u_J=(-1)^{\lvert I\cap J\rvert}u_J\cdot u_I$ where $u_I=\prod_{i\in I}u_i$ and $u_J=\prod_{i\in J}u_i$.
It follows that, for any $\chi_1,\chi_2\in\wh{G}$,
\[
s_{\chi_1}\cdot s_{\chi_2}=\hat{\gamma}(\chi_1,\chi_2)\,s_{\chi_2}\cdot s_{\chi_1},
\]
where
\begin{equation}\label{eq:gammahat_spin}
\hat{\gamma}(\chi_1,\chi_2)=(-1)^{\lvert \{1\leq i\leq q\,\vert\, \chi_1(g_i)=\chi_2(g_i)=-1\}\rvert}.
\end{equation}
By construction, the conjugation by $s_\chi$ is the automorphism $\Cl(u_\chi)$ of $\Cl(V)$, for all $\chi\in\wh{G}$, so the commutation factor $\hat{\gamma}$ is precisely $\Br(\omega_r)=[\Cl\subo(V)]$.

The following result, which will be used in Section \ref{s:D}, can be obtained by arguments similar to the above computation of $\hat\gamma$.

\begin{lemma}\label{lm:commutator}
Let $(V,Q)$ be a quadratic space of even dimension over a field of characteristic different from $2$. Let $u_1,u_2\in\Ort(V,Q)$ be commuting semisimple elements. Since $\Cl(V,Q)$ is central simple, there exist invertible elements $s_1,s_2\in\Cl(V,Q)$, unique up to scalar, such that $\Cl(u_i)(x)=s_i\cdot x\cdot s_i^{-1}$ for all $x\in\Cl(V,Q)$, $i=1,2$. Then the group commutator $[s_1,s_2]\bydef s_1\cdot s_2\cdot s_1^{-1}\cdot s_2^{-1}$ is given by $[s_1,s_2]=(-1)^{p_1 p_2+d}$ where $p_i=0$ if $u_i\in\SO(V,Q)$ and $p_i=1$ if $u_i\notin\SO(V,Q)$, $i=1,2$, and $d$ is the dimension of the common $(-1,-1)$-eigenspace of $u_1$, $u_2$.\qed
\end{lemma}

It is convenient to introduce some notation before we state our main result for Series B. As mentioned above, a $G$-grading on $\cL$ is determined by a multiset $\Xi=\{g_1,\ldots,g_{2r+1}\}$ satisfying \eqref{eq:relation_g0_trivT} with $g_0=e$. For $i=1,\ldots,q$, set
\[
\tilde{g}_i=g_1\cdots g_{i-1}g_{i+1}\cdots g_q.
\]
Then $\tilde{g}_i^2=e$ and $\tilde{g}_1\cdots\tilde{g}_q=e$. Consider the group homomorphism
\begin{equation}\label{eq:homom_B}
f_\Xi\colon\wh{G}\to\ZZ_2^q,\,\chi\mapsto(x_1,\ldots,x_q)\mbox{ where }\chi(\tilde{g}_i)=(-1)^{x_i}.
\end{equation}
Note that the image of $f_\Xi$ is contained in the hyperplane $(\ZZ_2^q)_0$ determined by the equation $x_1+\cdots+x_q=0$. Define $x\bullet y=\sum_{i=1}^q x_i y_i$ for all $x,y\in\ZZ_2^q$. Note that this is a symmetric bilinear form whose restriction to $(\ZZ_2^q)_0$ is alternating and nondegenerate.

\begin{theorem}\label{th:B}
Let $\cL$ be the simple Lie algebra of type $B_r$ ($r\ge 2$) over an algebraically closed field $\FF$ of characteristic $0$. Suppose $\cL$ is graded by an abelian group $G$. Then, for any dominant integral weight $\lambda=\sum_{i=1}^r m_i\omega_i$, we have $H_\lambda=\{e\}$ and $\Br(\lambda)=\hat{\gamma}^{m_r}$, where $\hat{\gamma}\colon\wh{G}\times\wh{G}\to\FF^\times$ is given by  $\hat{\gamma}(\chi_1,\chi_2)=(-1)^{f_\Xi(\chi_1)\bullet f_\Xi(\chi_2)}$ with $f_\Xi$ as in \eqref{eq:homom_B} associated to the parameter $\Xi$ of the grading on $\cL$.
\end{theorem}

\begin{proof}
Since $H_{\omega_i}$ is trivial for all $i$, we can apply Proposition \ref{prop:product}:
\[
\Br(\lambda)=\prod_{i=1}^r\Br(\omega_i)^{m_i}.
\]
But $\Br(\omega_i)=1$ for all $i<r$ and $\Br(\omega_r)=\hat{\gamma}$ according to \eqref{eq:gammahat_spin} in which we substitute $\tilde{g}_i=g_i(g_1\ldots g_q)$ for $g_i$.
\end{proof}

\begin{remark}\label{rem:support_spin}
The support of the graded division algebra representing $[\Br(\omega_r)]$ is the subgroup of $\langle\tilde{g}_1,\ldots,\tilde{g}_q\rangle$ given by
\[
\begin{split}
&\{\tilde{g}_1^{x_1}\cdots\tilde{g}_q^{x_q}\;|\;x\in\ZZ_2^q\mbox{ such that }x\bullet y=0\mbox{ for all }y\in\ZZ_2^q
\mbox{ satisfying }\tilde{g}_1^{y_1}\cdots\tilde{g}_q^{y_q}=e\}\\
&=\{ \tilde{g}_1^{x_1}\cdots \tilde{g}_q^{x_q}\;|\;x\in f_{\Xi}(\wh{G}) \}.
\end{split}
\]
Thus, it is an elementary $2$-group of (even) rank $\le q-1$.
\end{remark}

\begin{proof}
Let $T=\ZZ_2^q$. Define a homomorphism $\alpha\colon T\to G$ by $\alpha(e_i)=\tilde{g}_i$ where $\{e_1,\ldots,e_q\}$ is the standard basis of $\ZZ_2^q$. Let $H=\ker\alpha$. We can identify $\wh{T}$ with $T$ using the standard basis. Then $H^\perp\subset\wh{T}$ is precisely the image of $f_\Xi$. The result now follows from Lemma \ref{lm:beta_for_quotient} using $\beta(x,y)=\hat\beta(x,y)=(-1)^{x\bullet y}$ and regarding $\hat\gamma$ of Theorem \ref{th:B} as a bicharacter on the image of $f_\Xi$.

Alternatively, let $Q=\langle\tilde{g}_1,\ldots,\tilde{g}_q\rangle$, so $Q^\perp=\ker f_\Xi$. For any $x\in \ZZ_2^q$, set  $\tilde{g}^x\bydef \tilde g_1^{x_1}\cdots\tilde g_q^{x_q}$. Then, for any $\chi\in\wh{G}$ and $x\in\ZZ_2^q$,
\[
\chi(\tilde{g}^x) =(-1)^{x\bullet f_\Xi(\chi)}
\]
and, in particular,
\begin{equation}\label{eq:tildegxe}
\tilde{g}^x=e\ \text{if and only if}\ x\in f_\Xi(\wh{G})',
\end{equation}
where prime denotes the orthogonal subspace relative to the nondegenerate symmetric bilinear form $\bullet$ on $\ZZ_2^q$. Then, since $Q=(\ker f_\Xi)^\perp\subset (\rad\hat\gamma)^\perp$, we have
\[
\begin{split}
(\rad\hat\gamma)^\perp&=\{g\in Q\;|\; \chi(g)=1\ \forall\chi\in\wh{G}\mbox{ such that }\hat\gamma(\chi,\wh{G})=1\}\\
 &=\{g\in Q\;|\; \chi(g)=1\ \forall\chi\in\wh{G}\mbox{ such that }f_\Xi(\chi)\bullet f_\Xi(\wh{G})=0\}\\
 &=\{\tilde{g}^x\;|\; x\bullet f_\Xi(\chi)=0\ \forall\chi\in\wh{G}\mbox{ such that }f_\Xi(\chi)\bullet f_\Xi(\wh{G})=0\}\\
 &=\{\tilde{g}^x\;|\; x\bullet f_\Xi(\chi)=0\ \forall\chi\in\wh{G}\mbox{ such that }f_\Xi(\chi)\in f_\Xi(\wh{G})'\}\\
 &=\{\tilde{g}^x\;|\; x\in \bigl(f_\Xi(\wh{G})\cap f_\Xi(\wh{G})'\bigr)'\}\\
 &=\{\tilde{g}^x\;|\; x\in f_\Xi(\wh{G})+ f_\Xi(\wh{G})'\}\\
 &=\{\tilde{g}^x\;|\; x\in f_\Xi(\wh{G})\},
\end{split}
\]
where in the last equation we used \eqref{eq:tildegxe}.
\end{proof}

\begin{corollary}
The simple $\cL$-module $V_\lambda$ admits a $G$-grading making it a graded $\cL$-module if and only if $m_r$ is even or the elements $\tilde{g}_1,\ldots,\tilde{g}_q$ have the following property: for any $x\in f_\Xi(\wh{G})$, $\tilde{g}_1^{x_1}\cdots\tilde{g}_q^{x_q}=e$.
\end{corollary}

\section{Series C}\label{s:C}

Consider the simple Lie algebra $\cL=\Sp_{2r}(\FF)$ of type $C_{r}$, $r\ge 2$, where $\FF$ is an algebraically closed field of characteristic $0$. In this case all automorphisms of $\cL$ are inner. For a given $G$-grading on $\cL$, we are going to find the Brauer invariants of all simple $\cL$-modules.

In \cite{BK10}, the $G$-gradings on $\cL$ are classified in terms of the corresponding $G$-graded algebras $\cR=\End(V)$ equipped with a symplectic involution $\vphi$ where $V$ is the natural module of $\cL$, which has highest weight $\omega_1$ and dimension $n=2r$. Namely, the grading on $\cL$ corresponding to $(\cR,\vphi)$ comes from the identification of $\cL$ with the graded subspace $\sks(\cR,\vphi)=\{x\in\cR\;|\;\vphi(x)=-x\}$.

The existence of the involution $\vphi$ forces $T\subset G$ to be an elementary $2$-group, where $T$ is the support of the graded division algebra $\cD$ representing $[\cR]$ in the $G$-graded Brauer group.
Although the isomorphism class of a grading on $\cL$ is not determined by the parameters $(T,\beta)$ of $\cD$ (see \cite{BK10} or \cite[Theorem 3.69]{EKmon}), it turns out that these are all we need to obtain the Brauer invariants of all simple $\cL$-modules.
The reason is that, for any $i=1,\ldots,r$, the simple $\cL$-module of highest weight $\omega_i$ is contained with multiplicity $1$ in $\wedge^i V$ (see e.g. \cite[VIII, Section 13]{Bou7_9}, where it is shown that $\wedge^i V\cong V_{\omega_i}\oplus\wedge^{i-2}V$), so we can express its Brauer invariant in terms of that of $V$.

\begin{proposition}\label{prop:fund_C}
For the simple Lie algebra of type $C_r$ graded by an abelian group $G$ and for any $i=1,\ldots,r$, we have $\Br(\omega_i)=\Br(\omega_1)^i$ in the $G$-graded Brauer group.
\end{proposition}

\begin{proof}
Denote $\cR=\End(V)$ and $\rho\colon\cL\to\cR$ the natural representation. As in the proof of Proposition \ref{prop:wedge}, the algebra $\wt{\cR}=\End(V^{\ot i})$ is $G$-graded by identification with $\cR^{\ot i}$, and $\rho^{\ot i}\colon U(\cL)\to\wt{\cR}$ is a homomorphism of $G$-graded algebras. Moreover, identifying $\End(\wedge^i V)$ with $\veps\wt{\cR}\veps$ where $\veps\colon V^{\ot i}\to\wedge^i V$ is the standard projection, we have $\rho^{\wedge i}(a)=\veps\rho^{\ot i}(a)\veps$ for all $a\in U(\cL)$. Since $\veps$ is a homogeneous idempotent, $\rho^{\wedge i}\colon U(\cL)\to\veps\wt{\cR}\veps$ is a homomorphism of $G$-graded algebras. Let $\veps_0\in\veps\cR\veps$ be the projection of $\wedge^i V$ onto $V_{\omega_i}$ associated to the decomposition of $\wedge^i V$ into isotypic components (as an $\cL$-module). As in the proof of Proposition \ref{prop:product}, we see that $\veps_0$ is homogeneous and $\veps_0\wt{\cR}\veps_0\cong\End(V_{\omega_i})$ as $G$-graded algebras. Therefore,
\[
\Br(\omega_i)=[\End(V_{\omega_i})]=[\veps_0\wt{\cR}\veps_0]=[\wt{\cR}]=[\cR]^i=\Br(\omega_1)^i,
\]
where we have used Lemma \ref{lm:corner}.
\end{proof}

\begin{theorem}\label{th:C}
Let $\cL$ be the simple Lie algebra of type $C_r$ ($r\ge 2$) over an algebraically closed field $\FF$ of characteristic $0$. Suppose $\cL$ is graded by an abelian group $G$. Then, for any dominant integral weight $\lambda=\sum_{i=1}^r m_i\omega_i$, we have $H_\lambda=\{e\}$ and $\Br(\lambda)=\hat{\beta}^{\sum_{i=1}^{\lfloor (r+1)/2\rfloor} m_{2i-1}}$, where $\hat{\beta}\colon\wh{G}\times\wh{G}\to\FF^\times$ is the commutation factor associated to the parameters $(T,\beta)$ of the grading on $\cL$.
\end{theorem}

\begin{proof}
Since $H_{\omega_i}$ is trivial for all $i$, we can apply Proposition \ref{prop:product}:
\[
\Br(\lambda)=\prod_{i=1}^r\Br(\omega_i)^{m_i}.
\]
But by Proposition \ref{prop:fund_C}, we have $\Br(\omega_i)=\Br(\omega_1)^i=\hat{\beta}^i$. Also, $\hat{\beta}^2=1$. The result follows.
\end{proof}

\begin{corollary}
The simple $\cL$-module $V_\lambda$ admits a $G$-grading making it a graded $\cL$-module if and only if $T=\{e\}$ or the number $\sum_{i=1}^{\lfloor (r+1)/2\rfloor} m_{2i-1}$ is even.
\end{corollary}

\section{Series D}\label{s:D}

Consider the simple Lie algebra $\cL=\So_{2r}(\FF)$ of type $D_{r}$, $r\ge 3$, where $\FF$ is an algebraically closed field of characteristic $0$. The natural module $V$ of $\cL$ has highest weight $\omega_1$ and dimension $n=2r$. Suppose $\cL$ is given a $G$-grading. In the case $r=4$, assume that this is a matrix grading, i.e., the action of $\wh{G}$ fixes $\omega_1$. Our goal is to find the Brauer invariants of all simple $\cL$-modules.

\subsection{Preliminaries}

In \cite{BK10}, the $G$-gradings on $\cL$ are classified in terms of the corresponding $G$-graded algebras $\cR=\End(V)$ equipped with an orthogonal involution $\vphi$. Namely, the grading on $\cL$ corresponding to $(\cR,\vphi)$ comes from the identification of $\cL$ with the graded subspace $\sks(\cR,\vphi)=\{x\in\cR\;|\;\vphi(x)=-x\}$.

The existence of the involution $\vphi$ forces $T\subset G$ to be an elementary $2$-group, where $T$ is the support of the graded division algebra $\cD$ representing $[\cR]$ in the $G$-graded Brauer group. Moreover, for any $x\in\cR$, $\vphi(x)$ is the adjoint of $x$ with respect to a suitable nondegenerate $\FF$-bilinear form $B\colon W\times W\to\cD$ which is $\cD$-sesquilinear relative to an involution of the graded algebra $\cD$. Here, as in Subsection \ref{ss:A_I}, $W$ is a right vector space over $\cD$ such that $\cR=\End_\cD(W)$ as a $G$-graded algebra. Recall that $V$ can be identified with $W\ot_\cD N$ where $N$ is the natural left (ungraded) module for $\cD$, $\ell=\dim N$ is the graded Schur index of $V$ and $|T|=\ell^2$. Fix a standard matrix realization of $\cD$ as described at the beginning of Subsection \ref{ss:A_II}. Adjusting the sesquilinear form $B$, we may assume that the involution of $\cD$ is the matrix transpose. Selecting a suitable homogeneous $\cD$-basis $\{v_1,\ldots,v_k\}$ in $W$, $\deg v_i=g_i$, we may assume that
\begin{equation}\label{eq:relation_g0}
g_1^2t_1=\ldots=g_q^2t_q=g_{q+1}g_{q+2}=\ldots=g_{q+2s-1}g_{q+2s}=g_0^{-1},
\end{equation}
where $q+2s=k$, $g_0\in G$, $t_i\in T$, and the sesquilinear form $B$ is represented by the following block-diagonal matrix:
\begin{equation}\label{eq:PhiD}
\Phi=\diag\left(X_{t_1},\ldots,X_{t_q},\matr{0&I\\ I&0},\ldots,\matr{0&I\\ I&0}\right),
\end{equation}
where $I=X_e\in\cD$ and all $X_{t_i}$ are symmetric (see \cite{E09d} or \cite[Theorem 3.42]{EKmon}). The element $g_0$ has the meaning of the degree of $B$ as a linear map $W\ot W\to\cD$. For $r\ne 4$, the $G$-gradings on $\cL$ are classified by the graded division algebra $\cD$, the multiset $\Xi=\{g_1T,\ldots,g_kT\}$, and the element $g_0\in G$ (see \cite{BK10} or \cite[Theorem 3.74]{EKmon}). The same proof works for matrix gradings in the case $r=4$. Note, however, that this classification is up to isomorphism (matrix isomorphism in the case $r=4$), so one set of parameters may correspond to one or two $\inaut(\cL)$-orbits in the set of $G$-gradings on $\cL$.

For $i=1,\ldots,r-2$, the simple $\cL$-module of highest weight $\omega_i$ can be realized as $\wedge^i V$, whose Brauer invariant can be computed in the same way as in Proposition \ref{prop:wedge}. Thus, for $i<r-1$, we have $H_{\omega_i}=H_{\omega_1}=\{e\}$ and $\Br(\omega_i)=\Br(\omega_1)^i$ in the $G$-graded Brauer group. As to the simple $\cL$-modules $V_{\omega_{r-1}}$ and $V_{\omega_r}$, which are called the {\em half-spin modules}, they are fixed by the automorphisms of $\cL$ given by $x\mapsto uxu^{-1}$ with $u\in\SO(V,Q)$ and permuted by such automorphisms with $u\in\Ort(V,Q)\setminus\SO(V,Q)$. Here $Q$ is the quadratic form associated to the involution $\vphi$, namely, $Q(v)=\frac{1}{2}(v,v)_\Phi$ where the bilinear form $(\cdot,\cdot)_\Phi$ on $V=W\ot_\cD N$ is given by \eqref{eq:form_Phi}. To calculate $H_{\omega_{r-1}}=H_{\omega_r}$, it will be convenient to work with similitudes of $Q$. Recall that $f\in\GL(V)$ is a {\em similitude of multiplier} $\mu$ if $Q(f(v))=\mu Q(v)$ for all $v\in V$; $f$ is {\em proper} if $\det(f)=\mu^r$ and {\em improper} if $\det(f)=-\mu^r$. Then $x\mapsto fxf^{-1}$ is an automorphism of $\cL$; it fixes the half-spin modules if $f$ is proper and swaps them if $f$ is improper. For any $\chi\in\wh{G}$, the action of $\chi$ on $\cR$ is the conjugation by the matrix
\begin{equation}\label{eq:u_chi}
u_\chi=\diag\left(\chi(g_1),\ldots,\chi(g_k)\right)\otimes X_t,
\end{equation}
where the element $t\in T$ is such that $\chi(s)=\beta(t,s)$ for any $s\in T$. By Lemma \ref{lm:action_on_form} (with $G$ instead of $\bG$), $u_\chi$ is a similitude of multiplier $\chi(g_0)^{-1}$. On the other hand,
\[
\det(u_\chi)=\bigl(\chi(g_1)\cdots\chi(g_k)\bigr)^\ell\det(X_t)^k
\]
and $k\ell=n=2r$.

\begin{lemma}\label{lm:h}
For the similitude $u_\chi$ given by \eqref{eq:u_chi}, we have:
\begin{itemize}
\item If $|T|>4$, then $u_\chi$ is proper for any $\chi\in\wh{G}$.
\item If $|T|=4$, so $T=\{e,a,b,c=ab\}$, with $X_a$ and $X_b$ symmetric, then $u_\chi$ is proper if and only if $\chi(h)=1$, where
    \[
    h=\begin{cases}
    t_1\cdots t_q&\text{if $r$ is even,}\\
    ct_1\cdots t_q&\text{if $r$ is odd.}
    \end{cases}
    \]
\item If $|T|=1$, then $u_\chi$ is proper if and only if $\chi(h)=1$, where $h=g_0^rg_1\cdots g_k$. (Note that relations \eqref{eq:relation_g0} imply that $h$ has order at most $2$ and depends only on $g_1,\ldots,g_q$.)
\end{itemize}
\end{lemma}
\begin{proof}
If $|T|>4$, then $\det(X_t)=1$ for any $t$ by Lemma \ref{lm:det}, and
\[
\begin{split}
\bigl(\chi(g_1)\cdots\chi(g_k)\bigr)^2
&=\chi(g_1^2)\cdots\chi(g_q^2)\chi(g_{q+1}g_{q+2})^2\cdots\chi(g_{q+2s-1}g_{q+2s})^2\\
&=\chi(g_0)^{-(q+2s)}\chi(t_1\cdots t_q)=\chi(g_0)^{-k}\chi(t_1\cdots t_k),
\end{split}
\]
hence
\[
\det(u_\chi)=\bigl(\chi(g_1)\cdots\chi(g_k)\bigr)^{2\frac{\ell}{2}}= \chi(g_0)^{-k\frac{\ell}{2}}=\chi(g_0)^{-r}.
\]
This proves the first part.

If $|T|=4$, then $k=r$, and the computations above give
\[
\det(u_\chi)=\chi(g_0)^{-k}\chi(t_1\cdots t_q)\det(X_t)^k,
\]
so for even $k$, $u_\chi$ is proper if and only if $\chi(t_1\cdots t_q)=1$, while for odd $k$, $u_\chi$ is proper if and only if $\chi(t_1\cdots t_q)=\det(X_t)$, and, by Lemma \ref{lm:det}, $\det(X_t)=\beta(c,t)=\chi(c)$, whence the result.

If $|T|=1$, then $u_\chi$ is proper if and only if $\chi(g_1\cdots g_k)=\chi(g_0)^{-r}$, which is equivalent to
$\chi(g_0^r g_1\cdots g_k)=1$.
\end{proof}

Thus, if $|T|>4$, then the grading on $\cL$ is inner and $H_{\omega_{r-1}}=H_{\omega_r}=\{e\}$.
If $|T|=1$ or $4$, then the grading is inner if and only if $h=e$ where $h\in G$ is the element defined in Lemma \ref{lm:h}, also $H_{\omega_{r-1}}=H_{\omega_r}=\langle h\rangle$. We state for future reference:

\begin{df}\label{df:distinguished_element_D}
The \emph{distinguished element} of a $G$-grading on $\So_{2r}(\FF)$ is the element $h\in G$ of order at most $2$ defined by
\[
h=\begin{cases}
		g_0^rg_1\cdots g_k&\text{if $|T|=1$,}\\
    t_1\cdots t_q&\text{if $|T|=4$ and $r$ is even,}\\
    ct_1\cdots t_q&\text{if $|T|=4$ and $r$ is odd,}\\
    e&\text{if $|T|>4$.}
    \end{cases}
\]
\end{df}

The half-spin modules $V_{\omega_{r-1}}$ and $V_{\omega_r}$ can be constructed using the Clifford algebra $\Cl(V)=\Cl(V,Q)$ in the following way. The algebra $\Cl(V)$ is simple but its even part $\Cl\subo(V)$ is a direct sum of two simple ideals. The center of $\Cl\subo(V)$ is
\begin{equation}\label{eq:center}
Z\bigl(\Cl\subo(V,Q)\bigr)=\FF 1\oplus\FF z,
\end{equation}
where $z$ is the element, unique up to sign, in $\Spin(V,Q)$ that satisfies $z\cdot v=-v\cdot z$ for any $v\in V$. If $\{v_1,\ldots,v_{2r}\}$ is an orthogonal basis with $Q(v_i)=1$ for all $i$, then we may take $z=\pm v_1\cdots v_{2r}$. Thus, conjugation by $z$ gives the automorphism $\Upsilon$ of $\Cl(V)$ associated to the $\ZZ_2$-grading, i.e., $\Upsilon$ is $\id$ on $\Cl\subo(V)$ and $-\id$ on $\Cl\subuno(V)$. Note that $z^2=(-1)^r1$. Moreover, the center of $\Spin(V,Q)$ is
\[
Z\bigl(\Spin(V,Q)\bigr)=\{\pm 1,\pm z\}\cong\begin{cases} \ZZ_2\times\ZZ_2&\text{if $r$ is even,}\\ \ZZ_4&\text{if $r$ is odd.}
\end{cases}
\]

For even $r$, the elements $\veps_{\pm}\bydef\frac{1\pm z}{2}$ are central idempotents of $\Cl\subo(V)$, so $\Cl\subo(V)=\Cl\subo(V)\veps_+\oplus \Cl\subo(V)\veps_-$ is the decomposition of the even Clifford algebra as the direct sum of simple ideals. For odd $r$, the elements $\veps_{\pm}=\frac{1\mp\sqrt{-1}z}{2}$ are central idempotents of $\Cl\subo(V)$.

Let $S^\pm$ be the unique (up to isomorphism) irreducible module for $\Cl\subo(V)\veps_{\pm}$.
Denote by $\pi_\pm:\Cl\subo(V)\rightarrow \Cl\subo(V)\veps_\pm\cong\End(S^\pm)$, $x\mapsto x\veps_\pm$, the projections on each of the simple ideals. Hence,
\begin{equation}\label{eq:action_z}
\pi_\pm(z)=\begin{cases}
\pm\id&\text{if $r$ is even,}\\
\pm\sqrt{-1}\id&\text{if $r$ is odd.}
\end{cases}
\end{equation}
Recall that the Lie algebra $\cL=\So(V,Q)$ imbeds in $\Cl\subo(V,Q)$ as the subspace $[V,V]^\cdot$ (see \eqref{eq:[uv]sigmauv} and \eqref{eq:sigmauv} where we have to substitute the bilinear form $(\cdot,\cdot)_\Phi$ for $B$). This gives a Lie algebra homomorphism $\rho\colon\cL\rightarrow\brac{\Cl(V)}$. Composing $\rho$ with $\pi_\pm$, we get the two half-spin representations $\rho_\pm$, and we may identify $S^+$ with $V_{\omega_{r-1}}$ and $S^-$ with $V_{\omega_r}$ or vice versa, depending on the choice of the sign for $z$.

\subsection{$G$-grading on the even Clifford algebra}

We want to construct a $G$-grading on $\Cl\subo(V)$ such that $\rho$ is a homomorphism of $G$-graded algebras. Since $S^+$ and $S^-$ are non-isomorphic irreducible $\cL$-modules, the image of the extension $\rho\colon U(\cL)\to\Cl(V)$ is the entire $\Cl\subo(V)$, so the  $G$-grading on $\Cl\subo(V)$ will be uniquely determined by this property.

Given a similitude $f$ of multiplier $\mu$ on a quadratic space $(V,Q)$ and a square root $\mu^{1/2}$, we have $\mu^{-1/2}f\in \Ort(V,Q)$, and hence there is an automorphism $\Cl(\mu^{-1/2}f)$ of $\Cl(V,Q)$ that maps $v$ to $\mu^{-1/2}f(v)$ for all $v\in V$. Note that the restriction of $\Cl(\mu^{-1/2}f)$ to the even part $\Cl\subo(V,Q)$ does not depend on the choice of the square root. Thus we obtain a homomorphism from the group of similitudes $\GO(V,Q)$ to $\Aut(\Cl\subo(V,Q))$. In fact, this is a homomorphism of algebraic groups because it can be defined without using square roots in the following way. First, $\Cl\subo(V,Q)$ can be defined as the quotient $T(V\otimes V)/I_1(Q)+I_2(Q)$, where $I_1(Q)$ is the ideal generated by the elements $v\otimes v -Q(v)1$ for $v\in V$, and $I_2(Q)$ is the ideal generated by the elements $u\otimes v\otimes v\otimes w -Q(u)u\otimes w$, for $u,v,w\in V$ (see e.g. \cite[Lemma 8.1]{KMRT}). Then, for a similitude $f$ of multiplier $\mu$, the restriction of the above mapping $\Cl(\mu^{-1/2}f)$ to $V\ot V$ is given by $v\otimes w \mapsto \mu^{-1} f(v)\otimes f(w)$ for all $v,w\in V$. The extension of this latter to $T(V\otimes V)$ leaves $I_1(Q)$ and $I_2(Q)$ invariant and hence induces an automorphism of $\Cl\subo(V,Q)$, which coincides with $\Cl(\mu^{-1/2}f)|_{\Cl\subo(V,Q)}$.

For any $\chi\in\wh{G}$, $u_\chi$ is a similitude of multiplier $\chi(g_0)^{-1}$ (Lemma \ref{lm:action_on_form}), so we fix a square root of $\chi(g_0)$ and consider the automorphism $\tilde\alpha_\chi=\Cl\bigl(\chi(g_0)^{1/2}u_\chi\bigr)$ of $\Cl(V)$. Recall that the imbedding  $\rho\colon\cL\rightarrow\Cl(V)$ is defined by $\rho(\sigma_{u,v})=-\frac{1}{2}[u,v]^\cdot$ for all $u,v\in V$. As usual, let $\alpha_\chi$ be the action of $\chi$ on $\cL$.

\begin{lemma}
For any $\chi\in\wh{G}$, we have $\rho\alpha_\chi=\tilde\alpha_\chi\rho$.
\end{lemma}

\begin{proof}
For any $u,v,w\in V$, we compute:
\[
\begin{split}
u_\chi\sigma_{u,v}(w)&=(u,w)_\Phi\,u_\chi(v)-(v,w)_\Phi\,u_\chi(u)\\
 &=\chi(g_0)(u_\chi(u),u_\chi(w))_\Phi\,u_\chi(v)-\chi(g_0)(u_\chi(v),u_\chi(w))_\Phi\,u_\chi(u)\\
 &=\chi(g_0)\sigma_{u_\chi(u),u_\chi(v)}(u_\chi(w)),
\end{split}
\]
so we get $u_\chi\sigma_{u,v}u_\chi^{-1}=\chi(g_0)\sigma_{u_\chi(u),u_\chi(v)}$. Hence
\[
\begin{split}
\rho\alpha_\chi(\sigma_{u,v})&=\chi(g_0)\rho\bigl(\sigma_{u_\chi(u),u_\chi(v)}\bigr)
    =-\frac{1}{2}\chi(g_0)[u_\chi(u),u_\chi(v)]^\cdot\\
    &=-\frac{1}{2}[\tilde\alpha_\chi(u),\tilde\alpha_\chi(v)]^\cdot
    =\tilde\alpha_\chi\bigl(-\frac{1}{2}[u,v]^\cdot\bigr)\\
    &=\tilde\alpha_\chi\rho(\sigma_{u,v}),
\end{split}
\]
which proves the result.
\end{proof}

For any $\chi_1,\chi_2\in\wh{G}$, we have $u_{\chi_1}u_{\chi_2}=\hat\beta(\chi_1,\chi_2)u_{\chi_2}u_{\chi_1}$, where $\hat\beta=\Br(\omega_1)$ is the commutation factor of the natural module $V=V_{\omega_1}$, which takes values in $\{\pm 1\}$. Thus the restriction to $V$ of $\tilde\alpha_{\chi_1}\tilde\alpha_{\chi_2}$ equals the restriction to $V$ of $\hat\beta(\chi_1,\chi_2)\tilde\alpha_{\chi_2}\tilde\alpha_{\chi_1}$. Therefore, the commutator of $\tilde\alpha_{\chi_1}$ and $\tilde\alpha_{\chi_2}$ in the group $\Aut(\Cl(V))$ is given by
\begin{equation}\label{eq:commutator_alpha}
[\tilde\alpha_{\chi_1},\tilde\alpha_{\chi_2}]=
\begin{cases}
\id&\text{if $\hat\beta(\chi_1,\chi_2)=1$,}\\
\Upsilon& \text{if $\hat\beta(\chi_1,\chi_2)=-1$,}
\end{cases}
\end{equation}
where $\Upsilon$ is the automorphism of $\Cl(V)$ associated to the $\ZZ_2$-grading. Note that, in any case, the restrictions $\tilde\alpha_{\chi_1}\vert_{\Cl\subo(V)}$ and $\tilde\alpha_{\chi_2}\vert_{\Cl\subo(V)}$ commute. We claim that the mapping $\wh{G}\to\Aut(\Cl\subo(V))$ sending $\chi$ to $\tilde\alpha_\chi\vert_{\Cl\subo(V)}$ is a homomorphism of algebraic groups. Indeed, this mapping is the composition of two homomorphisms: $\wh{G}\to\PGO(V,Q)$ given by $\chi\mapsto\alpha_\chi$ and  $\PGO(V,Q)\to\Aut(\Cl\subo(V))$ induced by the homomorphism $\GO(V,Q)\to\Aut(\Cl\subo(V))$ defined above. Here we identified $\PGO(V,Q)\bydef\GO(V,Q)/\FF^\times$ with $\Aut(\cL)$ (a subgroup of $\Aut(\cL)$ in the case $r=4$).

\begin{corollary}
The homomorphism of algebraic groups $\wh{G}\to\Aut(\Cl\subo(V))$ given by
\[
\chi\mapsto\tilde\alpha_\chi\vert_{\Cl\subo(V)}
\]
endows $\Cl\subo(V)$ with a $G$-grading such that the imbedding $\rho\colon\cL\rightarrow \Cl\subo(V)^{(-)}$ is a homomorphism of $G$-graded algebras.\qed
\end{corollary}

\begin{remark}
If $|T|=1$, then $V$ itself is $G$-graded. For $u\in V_a$ and $v\in V_b$, the element $u\cdot v\in\Cl\subo(V)$ is homogeneous of degree $abg_0$. Note, however, that $\Cl(V)$ is not $G$-graded.
\end{remark}

We can now obtain relations between the Brauer invariants of the half-spin modules and that of the natural module. Let $H=H_{\omega_{r-1}}=H_{\omega_{r}}$. Then the $G$-grading on $\Cl\subo(V)\cong\End(S^+)\times\End(S^-)$ induces $(G/H)$-gradings on $\End(S^+)$ and $\End(S^-)$ such that $\rho_\pm\colon U(\cL)\to\End(S^\pm)$ are (surjective) homomorphisms of $(G/H)$-graded algebras. By definition, $\Br(S^\pm)=[\End(S^\pm)]$ in the $(G/H)$-graded Brauer group.

For any $u\in\Ort(V,Q)$, there exists an invertible element $s\in\Cl(V,Q)$, unique up to scalar, such that $\Cl(u)(x)=s\cdot x\cdot s ^{-1}$ for all $x\in\Cl(V,Q)$. It follows from $s\cdot V\cdot s^{-1}=V$ that the element $s$ must be even or odd; moreover, it can be normalized so that $s\cdot\tau(s)=1$ where $\tau$ is the canonical involution of $\Cl(V,Q)$. If $u\in\SO(V,Q)$ then $s$ is even, so it be taken in $\Spin(V,Q)$. Also, the conjugation by the element $\pi_\pm(s)=s\veps_\pm$ is the restriction of $\Cl(u)$ to the simple ideal $\Cl\subo(V,Q)\veps_\pm\cong\End(S^\pm)$. If $u\notin\SO(V,Q)$ then $s$ is odd and $\Cl(u)$ swaps the simple ideals $\Cl\subo(V,Q)\veps_\pm$.

\begin{lemma}\label{lm:commutator_s}
For any $\chi\in\wh{G}$, fix an invertible element $s_\chi\in\Cl(V)$ such that $\tilde{\alpha}_\chi$ is the conjugation by $s_\chi$. Then the group commutator $[s_{\chi_1},s_{\chi_2}]\bydef s_{\chi_1}\cdot s_{\chi_2}\cdot s_{\chi_1}^{-1}\cdot s_{\chi_2}^{-1}$ belongs to $\{\pm 1\}$ if $\hat\beta(\chi_1,\chi_2)=1$ and to $\{\pm z\}$ if $\hat\beta(\chi_1,\chi_2)=-1$, with $z$ as in \eqref{eq:center}.
\end{lemma}

\begin{proof}
The element $[s_{\chi_1},s_{\chi_2}]$ is even and independent of the choice of $s_\chi$. It follows that $[s_{\chi_1},s_{\chi_2}]\in\Spin(V,Q)$. By \eqref{eq:commutator_alpha}, the conjugation by $[s_{\chi_1},s_{\chi_2}]$ is $\id$ if $\hat\beta(\chi_1,\chi_2)=1$ and $\Upsilon$ if $\hat\beta(\chi_1,\chi_2)=-1$. The result follows.
\end{proof}

The following proposition is reminiscent of Theorem 9.12 in \cite{KMRT}, which deals with the ordinary Brauer group of the field $\FF$ (which is there of arbitrary characteristic and, of course, not assumed algebraically closed).

\begin{proposition}\label{prop:half_spin_relations}
Let the simple Lie algebra of type $D_r$ be graded by an abelian group $G$. If $r=4$, assume that the highest weight $\omega_1$ of the natural module is fixed by $\wh{G}$. Let $H=H_{\omega_{r-1}}=H_{\omega_r}$. Then we have the following relations in the $(G/H)$-graded Brauer group:
\begin{itemize}
\item
If $r$ is odd then $\Br(\omega_{r-1})\Br(\omega_r)=1$ and $\Br(\omega_{r-1})^2=\Br(\omega_r)^2=\Br(\omega_1)$.
\item
If $r$ is even then $\Br(\omega_{r-1})\Br(\omega_r)=\Br(\omega_1)$ and $\Br(\omega_{r-1})^2=\Br(\omega_r)^2=1$.
\end{itemize}
\end{proposition}

\begin{proof}
Let $K=H^\perp\subset\wh{G}$ and let $\hat\gamma_\pm\colon K\times K\to\FF^\times$ be the commutation factors of the half-spin modules $S^\pm$. For any $\chi\in K$, the element $s_\chi$ is even, so the conjugation by $\pi_\pm(s_\chi)=s_\chi\veps_\pm$ is the restriction of $\tilde{\alpha}_\chi$ to the simple ideal $\Cl\subo(V)\veps_\pm\cong\End(S^\pm)$. For $\chi_1,\chi_2\in K$, we have $\hat\gamma_\pm(\chi_1,\chi_2)=[\pi_\pm(s_{\chi_1}),\pi_\pm(s_{\chi_2})]$.
\begin{itemize}
\item If $\hat\beta(\chi_1,\chi_2)=1$, then Lemma \ref{lm:commutator_s} tells us that $[s_{\chi_1},s_{\chi_2}]$ is either $1$ or $-1$, hence
    \[
    \hat\gamma_+(\chi_1,\chi_2)=\hat\gamma_-(\chi_1,\chi_2)\in\{\pm 1\}.
    \]

\item If $\hat\beta(\chi_1,\chi_2)=-1$, then Lemma \ref{lm:commutator_s} tells us that $[s_{\chi_1},s_{\chi_2}]$ is either $z$ or $-z$. Taking into account \eqref{eq:action_z}, we see:
    \begin{itemize}
    \item If $r$ is even, then $z$ acts as $\id$ on $S^+$ and $-\id$ on $S^-$, hence
        \[
        \hat\gamma_+(\chi_1,\chi_2)=-\hat\gamma_-(\chi_1,\chi_2)\in\{\pm 1\}.
        \]
    \item If $r$ is odd, then $z$ acts as $\sqrt{-1}\id$ on $S^+$ and as $-\sqrt{-1}\id$ on $S^-$, hence
        \[
        \hat\gamma_+(\chi_1,\chi_2)=-\hat\gamma_-(\chi_1,\chi_2)\in\{\pm \sqrt{-1}\}.
        \]
    \end{itemize}
\end{itemize}
The result follows.
\end{proof}

\begin{corollary}\label{cor:equality_gamma_pm}
$\Br(\omega_{r-1})=\Br(\omega_r)$ if and only if $\Br(\omega_1)=1$.\qed
\end{corollary}

\subsection{Inner gradings on $\So_{2r}(\FF)$}\label{ss:D_I}

Assume that the given grading on $\cL=\So_{2r}(\FF)$ is inner, i.e., the similitudes $u_\chi$ are proper for all $\chi\in\wh{G}$. Recall that this happens if and only if $h=e$ where $h\in G$ is the distinguished element (see Definition \ref{df:distinguished_element_D}).

We will now determine the commutation factors $\gamma_\pm\colon\wh{G}\times\wh{G}\to\FF^\times$ of the half-spin modules $S^\pm$ of $\cL$. First we calculate $\hat\gamma_\pm(\chi_1,\chi_2)$ in a special case. Recall the quadratic form $\beta\colon T\to\{\pm 1\}$ defined by equation \eqref{df:quadratic_form_beta}.

\begin{lemma}\label{lm:commutator_s_continued}
For any $\chi\in\wh{G}$, fix a square root $\chi(g_0)^{1/2}$ and an invertible element $s_\chi\in\Cl(V)$ such that $\tilde{\alpha}_\chi=\Cl(\chi(g_0)^{1/2}u_\chi)$ is the conjugation by $s_\chi$. Consider $\chi,\psi\in\wh{G}$ such that the similitudes $u_{\chi}$ and $u_{\psi}$ are proper and $\beta(t')=\beta(t'')=\beta(t',t'')=1$ where $\chi|_T=\beta(t',\cdot)$ and $\psi|_T=\beta(t'',\cdot)$. Then the group commutator $[s_{\chi},s_{\psi}]$, which does not depend on the choice of the square roots, can be calculated as follows:
\begin{itemize}
\item If $|T|>16$, then $[s_{\chi},s_{\psi}]=1$.
\item If $|T|=16$, then $[s_{\chi},s_{\psi}]=\left\{\begin{array}{l} 1 \text{ if $\chi|_T=1$ or $\psi|_T=1$ or $\chi|_T=\psi|_T$,}\\
(-1)^{|\{1\le i\le q\;|\;\chi(t_i)=\psi(t_i)=1\}|} \text{ otherwise.}\end{array}\right.$
\item If $|T|=4$, then
\[
[s_{\chi},s_{\psi}]=\left\{\begin{array}{ll} 1 & \text{if $\chi|_T=1$ and $\psi|_T=1$,}\\
(-1)^{|\{1\le i\le q\;|\;\chi(t_i)=1,\,\mu_i=-1\}|} & \text{if $\chi|_T\ne 1$ and $\psi|_T=1$,}\\
(-1)^{|\{1\le i\le q\;|\;\psi(t_i)=1,\,\lambda_i=-1\}|} & \text{if $\chi|_T=1$ and $\psi|_T\ne 1$,}\\
(-1)^{|\{1\le i\le q\;|\;\chi(t_i)=\psi(t_i)=1,\,\lambda_i=\mu_i\}|} & \text{if $\chi|_T\ne 1$ and $\psi|_T\ne 1$,}\end{array}\right.
\]
where $\lambda_i=\chi(g_0)^{1/2}\chi(g_i)$ and $\mu_i=\psi(g_0)^{1/2}\psi(g_i)$ for $i=1,\ldots,q$.
\item If $|T|=1$, then $[s_{\chi},s_{\psi}]=(-1)^{|\{1\le i\le q\;|\;\lambda_i=\mu_i=-1\}|}$, with $\lambda_i$ and $\mu_i$ as above.
\end{itemize}
\end{lemma}

\begin{proof}
First of all, changing the square root $\chi(g_0)^{1/2}$ is tantamaunt to replacing $s_\chi$ by $zs_\chi$, but $[zs_\chi,s_\psi]=[s_\chi,s_\psi]$ since $s_\psi$ is even.
Recall that
\[
u_\chi=\diag\left(\chi(g_1),\ldots,\chi(g_k)\right)\otimes X_{t'}\quad\mbox{and}\quad
u_\psi=\diag\left(\psi(g_1),\ldots,\psi(g_k)\right)\otimes X_{t''}.
\]
Since $\langle t',t''\rangle$ is a totally isotropic subspace of $T$ with respect to the quadratic form $\beta(\cdot)$, Witt Extension Theorem implies that we may choose totally isotropic subspaces $\wt{A}$ and $\wt{B}$ such that $T=\wt{A}\times\wt{B}$ and $t',t''\in\wt{A}$. Then we can pick a symplectic basis $\{\tilde{a}_j\}\cup\{\tilde{b}_j\}$ of $T$ with $\wt{A}=\langle\tilde{a}_j\rangle$ and $\wt{B}=\langle\tilde{b}_j\rangle$. We use tildes to distinguish from the symplectic basis associated to the standard matrix realization of $\cD$ leading to equations \eqref{eq:relation_g0} and \eqref{eq:PhiD}, which was fixed earlier. Since $X_{a_j}^2=1$, $X_{b_j}^2=1$ and $X_{a_i}X_{b_j}=(-1)^{\delta_{ij}}X_{b_j}X_{a_j}$ (Kronecker delta) are defining relations for $\cD$, the mapping $X_{a_j}\mapsto X_{\tilde{a}_j}$, $X_{b_j}\mapsto X_{\tilde{b}_j}$ extends to an automorphism of $\cD$ and thus leads to a new matrix realization of $\cD$, in which the elements $X_{t'}$ and $X_{t''}$ are diagonal. Namely, using the notation $N=\FF\wt{B}$ as in Remark \ref{rem:matrix_realization}, we have $X_{t'}\tilde{e}_b=\beta(t',b)\tilde{e}_b$ and $X_{t''}\tilde{e}_b=\beta(t'',b)\tilde{e}_b$ for all $b\in\wt{B}$. Therefore, the basis $v_{i,b}\bydef v_i\ot \tilde{e}_b$ ($i=1,\ldots,k$, $b\in\wt{B}$) of $V=W\ot_\cD N$ consists of common eigenvectors of $u_\chi$ and $u_\psi$. Setting $\lambda_i=\chi(g_0)^{1/2}\chi(g_i)$ and $\mu_i=\psi(g_0)^{1/2}\psi(g_i)$ for $i=1,\ldots,k$, we obtain:
\begin{align*}
\chi(g_0)^{1/2}u_\chi(v_{i,b})&=\lambda_{i,b}v_{i,b}\text{ where }\lambda_{i,b}=\lambda_i\beta(t',b)=\lambda_i\chi(b),\\
\psi(g_0)^{1/2}u_\psi(v_{i,b})&=\mu_{i,b}v_{i,b}\text{ where }\mu_{i,b}=\mu_i\beta(t'',b)=\mu_i\psi(b).
\end{align*}
Relations \eqref{eq:relation_g0} imply that $\lambda_i^2=\chi(t_i)$ for $i=1,\ldots,q$ and $\lambda_{q+2j-1}\lambda_{q+2j}=1$ for $j=1,\ldots,s$, and similarly for the $\mu_i$. It follows that
\begin{align*}
&\lambda_{i,b}^2=\chi(t_i),\,\mu_{i,b}^2=\psi(t_i)\text{ for }1\le i\le q,\,b\in\wt{B};\\
&\lambda_{q+2j-1,b}\lambda_{q+2j,b}=\mu_{q+2j-1,b}\mu_{q+2j,b}=1\text{ for }1\le j\le s,\,b\in\wt{B}.
\end{align*}
By Lemma \ref{lm:commutator}, we have $[s_\chi,s_\psi]=(-1)^d$ where $d$ is the dimension of the common $(-1,-1)$-eigenspace $V_{-1}$ of the proper isometries $\chi(g_0)^{1/2}u_\chi$ and $\psi(g_0)^{1/2}u_\psi$. Note that $v_{q+2j-1,b}\in V_{-1}$ if and only if $v_{q+2j,b}\in V_{-1}$, so the vectors $v_{i,b}$ with $i>q$ do not affect the parity of $d$. Also, if $\chi(t_i)=-1$ or $\psi(t_i)=-1$ for some $i\le q$ then $v_{i,b}\notin V_{-1}$ for all $b\in B$. Therefore, we can restrict our attention to $1\le i\le q$ such that $\chi(t_i)=\psi(t_i)=1$. For any such $i$, we have $v_{i,b}\in V_{-1}$ if and only if $\chi(b)=-\lambda_i$ and $\psi(b)=-\mu_i$.

If $|T|=1$, then $\wt{B}=\{e\}$, so $v_{i,e}\in V_{-1}$ if and only if $\lambda_i=\mu_i=-1$.
Otherwise $\ell=|\wt{B}|$ is a nontrivial power of $2$. Note that, for given $\xi,\eta\in\{\pm 1\}$, the set $\wt{B}(\xi,\eta)\bydef\{b\in \wt{B}\;|\;\chi(b)=\xi,\,\psi(b)=\eta\}$ has size $0$ or $\ell/2$ or $\ell$ unless the restrictions $\chi|_{\wt{B}}$ and $\psi|_{\wt{B}}$ are distinct and nontrivial. In this latter case (which is impossible for $\ell=2$), the set $\wt{B}(\xi,\eta)$ has size $\ell/4$. If $\ell>4$ then all of these numbers are even, hence so is $d$. If $\ell=4$ then we get size $1$ if $\chi|_{\wt{B}}$ and $\psi|_{\wt{B}}$ are distinct and nontrivial, and even size otherwise. Finally, consider $\ell=2$. We get even size unless at least one of $\chi|_{\wt{B}}$ and $\psi|_{\wt{B}}$ is nontrivial. If, say, $\chi|_{\wt{B}}\ne 1$ then $|\wt{B}(\xi,\eta)|=1$ if and only if $\psi|_{\wt{B}}=1$, $\eta=1$ or $\psi|_{\wt{B}}=\chi|_{\wt{B}}$, $\eta=\xi$. It remains to observe that, since $t',t''\in\wt{A}$, we have $\chi(\wt{A})=1$ and $\psi(\wt{A})=1$, so we may replace the restrictions of $\chi|_{\wt{B}}$ and $\psi|_{\wt{B}}$ by $\chi|_T$ and $\psi|_T$, respectively, in the above conditions.
\end{proof}

\medskip

$\blacklozenge$ Consider the case $|T|=1$.

The $G$-grading on $\cL$ is induced by a $G$-grading on the natural module $V=V_{\omega_1}$ and is determined by $g_0\in G$ and a multiset $\Xi=\{g_1,\ldots,g_{2r}\}$ satisfying \eqref{eq:relation_g0_trivT}.
Note that $q$ is even. Let $\mathbf{1}=(1,\ldots,1)\in\ZZ_2^q$ and consider the map
\begin{equation}\label{eq:homom_D}
f_{\Xi,g_0}\colon\wh{G}\to\ZZ_2^q/\langle\mathbf{1}\rangle,\,\chi\mapsto(x_1,\ldots,x_q)+\langle\mathbf{1}\rangle\mbox{ where }\chi(g_0)^{1/2}\chi(g_i)=(-1)^{x_i}.
\end{equation}
Note that this map does not depend on the choice of the square roots and is a homomorphism of groups. Since $h=g_0^{q/2} g_1\cdots g_q=e$, the image of $f_{\Xi,g_0}$ is contained in $(\ZZ_2^q)_0/\langle\mathbf{1}\rangle$ where $(\ZZ_2^q)_0$ is the hyperplane determined by the equation $x_1+\cdots+x_q=0$. Define $x\bullet y=\sum_{i=1}^q x_i y_i$ for all $x,y\in\ZZ_2^q$. Note that this is a symmetric bilinear form whose restriction to $(\ZZ_2^q)_0$ is alternating and degenerate. The radical of this restriction is precisely $\langle\mathbf{1}\rangle$, so we obtain an alternating nondegenerate form on $(\ZZ_2^q)_0/\langle\mathbf{1}\rangle$, which we still denote by $x\bullet y$.
By Lemma \ref{lm:commutator_s_continued}, we obtain:
\begin{equation}\label{eq:gammahat_T1}
\hat{\gamma}_+(\chi_1,\chi_2)=\hat{\gamma}_-(\chi_1,\chi_2)=(-1)^{f_{\Xi,g_0}(\chi_1)\bullet f_{\Xi,g_0}(\chi_2)}.
\end{equation}
The following is analogous to Remark \ref{rem:support_spin}:

\begin{remark}
If $q=0$ then $\Br(\omega_{r-1})=\Br(\omega_r)=1$. Otherwise $g_0$ is a square in $G$, so we may shift the grading on $V$ and assume $g_0=e$. Then we can define $f_\Xi\colon\wh{G}\to\ZZ_2^q$, $\chi\mapsto (x_1,\ldots,x_q)$, where $\chi(g_i)=(-1)^{x_i}$, and the support of the graded division algebra representing $[\Br(\omega_{r-1})]=[\Br(\omega_r)]$ is the subgroup of $\langle g_1,\ldots,g_q\rangle$ given by
\[
\begin{split}
&\{g_1^{x_1}\cdots g_q^{x_q}\;|\;x\in\ZZ_2^q\mbox{ such that }x\bullet y=0\mbox{ for all }y\in\ZZ_2^q
\mbox{ satisfying }g_1^{y_1}\cdots g_q^{y_q}=e\}\\
&=\{ g_1^{x_1}\cdots g_q^{x_q}\;|\;x\in f_{\Xi}(\wh{G}) \}.
\end{split}
\]
Thus, it is an elementary $2$-group of (even) rank $\le q-2$.
\end{remark}

\medskip

$\blacklozenge$ Consider the case $|T|=4$.

Here $T=\{e,a,b,c\}$, $\beta(a)=\beta(b)=1$ (i.e., $X_a$ and $X_b$ are symmetric). The $G$-grading on $\cL$ is determined by $g_0\in G$ and a multiset $\Xi=\{g_1,\ldots,g_r\}$ satisfying \eqref{eq:relation_g0} with $k=r$. Note that $q$ and $r$ have the same parity. Also note that $t_i\in\{e,a,b\}$ for all $i=1,\ldots,q$.

Fix $\chi_a,\chi_b\in\wh{G}$ such that $\chi_a(t)=\beta(a,t)$ and $\chi_b(t)=\beta(b,t)$ for all $t\in T$. Set $\chi_e=1$ and $\chi_c=\chi_a\chi_b$. Then $\{\chi_e,\chi_a,\chi_b,\chi_c\}$ is a transversal for the subgroup $T^\perp$ of $\wh{G}$.

Since $\hat\beta(\chi_a,\chi_b)=\beta(a,b)=-1$, Lemma \ref{lm:commutator_s} tells us that the group commutator of the corresponding invertible elements $s_{\chi_a}$ and $s_{\chi_b}$ of $\Cl\subo(V)$ is a nonscalar central element of the spin group, so we may set
\[
z\bydef [s_{\chi_a},s_{\chi_b}].
\]
Then we have, by \eqref{eq:action_z}, that $\hat\gamma_\pm(\chi_a,\chi_b)=[\pi_\pm(s_{\chi_a}),\pi_\pm(s_{\chi_b})]=\pm 1$ if $r$ is even and $\pm\bi$ if $r$ is odd, where $\bi=\sqrt{-1}$.

For any $\psi\in T^\perp$, we have $\hat\beta(\chi_a,\psi)=\hat\beta(\chi_b,\psi)=1$, so Lemma \ref{lm:commutator_s_continued} yields
\begin{align*}
\hat\gamma_\pm(\chi_a,\psi)&=[\pi_\pm(s_{\chi_a}),\pi_\pm(s_\psi)]=(-1)^{|\{1\le i\le q\;|\;t_i\in\{e,a\},\,\mu_i=-1\}|},\\
\hat\gamma_\pm(\chi_b,\psi)&=[\pi_\pm(s_{\chi_b}),\pi_\pm(s_\psi)]=(-1)^{|\{1\le i\le q\;|\;t_i\in\{e,b\},\,\mu_i=-1\}|},
\end{align*}
where $\mu_i=\psi(g_0)^{1/2}\psi(g_i)$. For any $t\in T$, define
\[
I_t=\{1\le i\le q\;|\;t_i=t\}.
\]
Then $I_c=\emptyset$ and the sets $I_e$, $I_a$, $I_b$ form a partition of $\{1,\ldots,q\}$. Now, the distinguished element $h$ is given by
\[
h=\begin{cases}
a^{|I_a|}b^{|I_b|}&\text{if $r$ is even,}\\
a^{|I_a|+1}b^{|I_b|+1}&\text{if $r$ is odd.}
\end{cases}
\]
Since $h=e$, we conclude that the sizes of $I_a$ and $I_b$ have the same parity as $r$. It follows that the same is true of $I_e$. Set
\[
g_a=g_0^{(|I_e|+|I_a|)/2}\prod_{i\in I_e\cup I_a}g_i\quad\mbox{and}\quad g_b=g_0^{(|I_e|+|I_b|)/2}\prod_{i\in I_e\cup I_b}g_i.
\]
Note that $g_a^2=e$, $g_b^2=e$ if $r$ is even and $g_a^2=a$, $g_b^2=b$ if $r$ is odd. Set $g_e=e$ and $g_c=g_a g_b$. Then, for any $\psi\in T^\perp$, we have $\hat\gamma(\chi_e,\psi)=\psi(g_e)$ and we can restate the above formulas as
\[
\hat\gamma_\pm(\chi_a,\psi)=\psi(g_a)\quad\mbox{and}\quad\hat\gamma_\pm(\chi_b,\psi)=\psi(g_b).
\]
It follows that
\[
\hat\gamma_\pm(\chi_c,\psi)=\hat\gamma_\pm(\chi_a,\psi)\hat\gamma_\pm(\chi_b,\psi)=\psi(g_a)\psi(g_b)=\psi(g_c).
\]
Also note that, for any $\psi',\psi''\in T^\perp$, we have $\hat\gamma_\pm(\psi',\psi'')=1$ by Lemma \ref{lm:commutator_s_continued}.

\begin{itemize}
\item If $r$ is even, then we have $\hat\gamma_+(\chi_a,\chi_b)=1$ and hence also $\hat\gamma_+(\chi_a,\chi_c)=1$ and $\hat\gamma_+(\chi_c,\chi_b)=1$. It follows that
\begin{equation}\label{eq:gammahat_T4_even}
\hat\gamma_+(\psi'\chi_{t'},\psi''\chi_{t''})=\psi'(g_{t''})\psi''(g_{t'}),
\end{equation}
for all $t',t''\in T$ and $\psi',\psi''\in T^\perp$.
Since every element of $\wh{G}$ can be written uniquely in the form $\psi\chi_t$ ($\psi\in T^\perp$, $t\in T$), the above formula determines $\hat\gamma_+(\chi_1,\chi_2)$ for all $\chi_1,\chi_2\in\wh{G}$. Note that $\hat\gamma_-=\hat\beta\hat\gamma_+$ by Proposition \ref{prop:half_spin_relations}.

\item If $r$ is odd, then we have $\hat\gamma_+(\chi_a,\chi_b)=\bi$ and hence also $\hat\gamma_+(\chi_a,\chi_c)=\bi$ and $\hat\gamma_+(\chi_c,\chi_b)=\bi$. It follows that
\begin{equation}\label{eq:gammahat_T4_odd}
\hat\gamma_+(\psi'\chi_{t'},\psi''\chi_{t''})=\beta^{1/2}(t',t'')\psi'(g_{t''})\psi''(g_{t'}),
\end{equation}
for all $t',t''\in T$ and $\psi',\psi''\in T^\perp$, where
\[
\beta^{1/2}(t',t'')\bydef
\begin{cases}
\phantom{+}1 & \text{if $\beta(t',t'')=1$,}\\
\phantom{+}\bi & \text{if $\beta(t',t'')=-1$ and $t'<t''$,}\\
-\bi & \text{if $\beta(t',t'')=-1$ and $t'>t''$,}
\end{cases}
\]
and the nontrivial elements of $T$ are formally ordered $a<c<b$. This determines $\hat\gamma_+(\chi_1,\chi_2)$ for all $\chi_1,\chi_2\in\wh{G}$, and $\hat\gamma_-=\hat\gamma_+^{-1}$by Proposition \ref{prop:half_spin_relations}.
\end{itemize}

\begin{remark}
Let $Q=\langle T,g_a,g_b\rangle$. If $r$ is odd then $\psi\chi_t$ is in $\rad\hat\gamma_+$ if and only if $t=e$ and $\psi(g_a)=\psi(g_b)=1$, hence $Q$ is the support of the graded division algebra representing $\hat\gamma_+$ and $\hat\gamma_-$, $Q\cong\ZZ_4^2$. If $r$ is even, then $\psi\chi_t$ is in $\rad\hat\gamma_+$ if and only if $g_t\in T$ and $\psi(g_a)=\psi(g_b)=1$, hence the support of the graded division algebra representing $\hat\gamma_+$ is the subgroup of $Q$ consisting of all elements $x\in Q$ satisfying $\chi_a(x)=1$ if $g_a\in T$, $\chi_b(x)=1$ if $g_b\in T$ and $\chi_a(x)=\chi_b(x)$ if $g_a g_b\in T$; this is an elementary $2$-group of rank $0$, $2$ or $4$.
\end{remark}

\medskip

$\blacklozenge$ Consider the case $|T|=16$ (hence $r$ is even).

Here $T=\langle a_1,a_2\rangle\times\langle b_1,b_2\rangle$, $\beta(a_j)=\beta(b_j)=1$ for $j=1,2$. The $G$-grading on $\cL$ is determined by $g_0\in G$ and a multiset $\Xi=\{g_1,\ldots,g_{k}\}$ satisfying \eqref{eq:relation_g0} with $k=r/2$, where $\beta(t_i)=1$ for all $i=1,\ldots,q$. Note that $q$ and $r/2$ have the same parity.

Pick elements $\chi_{a_j},\chi_{b_j}$ of $\wh{G}$, $j=1,2$, such that $\chi_{a_j}(t)=\beta(a_j,t)$ and $\chi_{b_j}(t)=\beta(b_j,t)$ for all $t\in T$. It follows from Lemma \ref{lm:commutator_s_continued} that $T^\perp$ lies in the radicals of $\hat\gamma_+$ and $\hat\gamma_-$. Indeed, if $\psi',\psi''\in T^\perp$ then $\hat\gamma_\pm(\psi',\psi'')=[\pi_\pm(s_{\psi'}),\pi_\pm(s_{\psi''})]=1$. Also,  $\hat\gamma_\pm(\chi_{a_j},\psi)=1$ and $\hat\gamma_\pm(\chi_{b_j},\psi)=1$ for any $\psi\in T^\perp$. Since $T^\perp$ together with $\chi_{a_j}$ and $\chi_{b_j}$, $j=1,2$, generate the group $\wh{G}$, we conclude that $\psi$ is in $\rad\hat\gamma_\pm$, as claimed.

Define a homomorphism
\[
f\colon\wh{G}\to\ZZ_2^4,\, \chi\mapsto (x_1,x_2,y_1,y_2)\mbox{ where }\chi(a_j)=(-1)^{x_j},\,\chi(b_j)=(-1)^{y_j},\,j=1,2.
\]
The kernel of $f$ is precisely $T^\perp$, and we have just shown that $\hat\gamma_\pm$ factors through $f$. Note that $\beta(\cdot)$ is given by
\[
\beta(t)=(-1)^{x_1y_1+x_2y_2}\mbox{ where }t=a_1^{x_1}a_2^{x_2}b_1^{y_1}b_2^{y_2}.
\]
Also, $\chi_{b_j}(t)=x_j$ and $\chi_{a_j}(t)=y_j$, $j=1,2$.

Now write $t_i=a_1^{x^{(i)}_1} a_2^{x^{(i)}_2} b_1^{y^{(i)}_1} b_2^{y^{(i)}_2}$, $i=1,\ldots,q$.
Then Lemma \ref{lm:commutator_s_continued} yields
\[
\begin{array}{ll}
\hat\gamma_\pm(\chi_{b_1},\chi_{b_2})=(-1)^{\sum_{i=1}^q(x_1^{(i)}+1)(x^{(i)}_2+1)}, & \hat\gamma_\pm(\chi_{b_1},\chi_{a_2})=(-1)^{\sum_{i=1}^q(x_1^{(i)}+1)(y^{(i)}_2+1)},\\
\hat\gamma_\pm(\chi_{b_2},\chi_{a_1})=(-1)^{\sum_{i=1}^q(x_2^{(i)}+1)(y^{(i)}_1+1)}, & \hat\gamma_\pm(\chi_{a_1},\chi_{a_2})=(-1)^{\sum_{i=1}^q(y_1^{(i)}+1)(y^{(i)}_2+1)}.
\end{array}
\]
Let $\chi'=\chi_{b_1}$ and $\chi''=\chi_{a_1}$. Since $\hat\beta(\chi',\chi'')=\beta(b_1,a_1)=-1$, Lemma \ref{lm:commutator_s} tells us that the group commutator of the corresponding invertible elements $s_{\chi'}$ and $s_{\chi''}$ of $\Cl\subo(V)$ is a nonscalar central element of the spin group, so we may set
\[
z\bydef [s_{\chi'},s_{\chi''}].
\]
Then we have, by \eqref{eq:action_z}, that $\hat\gamma_+(\chi_{b_1},\chi_{a_1})=[\pi_+(s_{\chi_a}),\pi_+(s_{\chi_b})]=1$. Finally, using Lemma \ref{lm:commutator_s_continued} again, we compute:
\[
\begin{split}
\hat\gamma_+(\chi_{b_2},\chi_{a_2})&=
\hat\gamma_+(\chi_{b_1}\chi_{b_2},\chi_{a_1}\chi_{a_2})
\hat\gamma_+^{-1}(\chi_{b_1},\chi_{a_1})
\hat\gamma_+^{-1}(\chi_{b_1},\chi_{a_2})
\hat\gamma_+^{-1}(\chi_{b_2},\chi_{a_1})\\
&=(-1)^{\sum_{i=1}^q\big((x_1^{(i)}+x_2^{(i)}+1)(y^{(i)}_1+y^{(i)}_2+1)+(x_1^{(i)}+1)(y^{(i)}_2+1)+(x_2^{(i)}+1)(y^{(i)}_1+1)\big)}\\
&=(-1)^{\sum_{i=1}^q(x_1^{(i)}y_1^{(i)}+x_2^{(i)}y_2^{(i)}+1)}=(-1)^q\prod_{i=1}^q\beta(t_i)=(-1)^q.
\end{split}
\]
To summarize our calculations, define a $4\times 4$ matrix with entries in $\ZZ_2$ as follows:
\[
M^+_{\Xi,g_0}=\sum_{i=1}^q M^+(t_i),
\]
where, for any $t=a_1^{x_1}a_2^{x_2}b_1^{y_1}b_2^{y_2}$ with $\beta(t)=1$, the symmetric matrix $M^+(t)$ is
\[
M^+(t)=
\begin{bmatrix}
0          & (x_1+1)(x_2+1) & 0              & (x_1+1)(y_2+1)\\
           & 0              & (x_2+1)(y_1+1) & 1\\
           &                & 0              & (y_1+1)(y_2+1)\\
\text{sym} &                &                & 0
\end{bmatrix}.
\]
Then, for any $\chi_1,\chi_2\in\wh{G}$, $\hat\gamma_+(\chi_1,\chi_2)$ is given by
\begin{equation}\label{eq:gammahat_T16}
\hat\gamma_+(\chi_1,\chi_2)=(-1)^{{}^t(f(\chi_1))M^+_{\Xi,g_0}f(\chi_2)}.
\end{equation}
Finally, $\hat\gamma_-=\hat\beta\hat\gamma_+$ by Proposition \ref{prop:half_spin_relations}.

\begin{remark}
The supports of the graded division algebras representing $\hat\gamma_+$ and $\hat\gamma_-$ are contained in $T$, so each of them is an elementary $2$-group of rank $0$, $2$ or $4$.
\end{remark}

\medskip

$\blacklozenge$ Consider the case $|T|>16$ (hence $r$ is even).

Say, $|T|=2^{2m}$ with $m>2$, so $T=\langle a_1,\ldots,a_m\rangle\times\langle b_1,\ldots,b_m\rangle$, $\beta(a_j)=\beta(b_j)=1$ for $j=1,\ldots,m$. Using the same notation as in the previous case, we obtain from Lemma \ref{lm:commutator_s_continued} that $T^\perp\subset\rad\hat\gamma_\pm$ and $\hat\gamma_\pm(\chi_u,\chi_v)=1$ where $u$ and $v$ are any of the elements $a_j$, $b_j$ except for  the cases $u=a_j$, $v=b_j$ or $u=b_j$, $v=a_j$. Defining $z$ appropriately, we may assume that $\hat\gamma_+(\chi_{a_1},\chi_{b_1})=1$. But then, using Lemma \ref{lm:commutator_s_continued} again, we compute for $j>1$:
\[
\hat\gamma_+(\chi_{a_j},\chi_{b_j})=
\hat\gamma_+(\chi_{a_1}\chi_{a_j},\chi_{b_1}\chi_{b_j})
\hat\gamma_+^{-1}(\chi_{a_1},\chi_{b_1})
\hat\gamma_+^{-1}(\chi_{a_1},\chi_{b_j})
\hat\gamma_+^{-1}(\chi_{a_j},\chi_{b_1})=1.
\]
Therefore, for all $\chi_1,\chi_2\in\wh{G}$, we obtain:
\begin{equation}\label{eq:gammahat_Tmore}
\hat\gamma_+(\chi_1,\chi_2)=1\quad\mbox{and}\quad\hat\gamma_-(\chi_1,\chi_2)=\hat\beta(\chi_1,\chi_2).
\end{equation}

\medskip

By examining the above cases, we see that if $|T|>1$ then $\hat\gamma_+\ne\hat\beta$. Recall also that $\hat\gamma_+\ne\hat\gamma_-$ if and only if $|T|>1$ (Corollary \ref{cor:equality_gamma_pm}). In summary:

\begin{theorem}\label{th:D_inner_classification}
Let $\cL=\So(V)$ be the simple Lie algebra of type $D_r$ ($r\ge 3$) over an algebraically closed field $\FF$ of characteristic $0$. Suppose $\cL$ is given a grading $\Gamma$ by an abelian group $G$, with parameters $(T,\beta)$, $\Xi$ and $g_0$, such that the image of $\wh{G}$ in $\Aut(\cL)$ consists of inner automorphisms. If $|T|>1$ then the $\Aut(\cL)$-orbit (orbit under the stabilizer of the natural module if $r=4$) of $\Gamma$ consists of two $\inaut(\cL)$-orbits, which are distinguished by the Brauer invariants of the half-spin modules, $\hat\gamma_+$ and $\hat\gamma_-$, where $\hat\gamma_+$ is given by equations \eqref{eq:gammahat_T4_even} through \eqref{eq:gammahat_Tmore}, while $\hat\gamma_-=\hat\gamma_+^{-1}$ for odd $r$ and $\hat\gamma_-=\hat\beta\hat\gamma_+$ for even $r$, with  $\hat{\beta}\colon\wh{G}\times\wh{G}\to\FF^\times$ being the commutation factor associated to the parameters $(T,\beta)$. If $|T|=1$ then $\hat\gamma_+=\hat\gamma_-$ are given by equation \eqref{eq:gammahat_T1}, but we may have one or two $\inaut(\cL)$-orbits: the grading $\Gamma$ on $\cL$ in this case is induced by a $G$-grading on the natural module $V$, and we get one $\inaut(\cL)$-orbit if and only if there exists an improper isometry $V\to V$ that is homogeneous of some degree with respect to the $G$-grading.\qed
\end{theorem}

\begin{theorem}\label{th:D_inner}
Under the conditions of Theorem \ref{th:D_inner_classification}, consider a dominant integral weight $\lambda=\sum_{i=1}^r m_i\omega_i$.
If $|T|>1$, denote by $\Gamma_+$ a representative of the $\inaut(\cL)$-orbit for which $\Br(\omega_{r-1})=\hat\gamma_-$, $\Br(\omega_r)=\hat\gamma_+$ and by $\Gamma_-$ a representative of the orbit for which $\Br(\omega_{r-1})=\hat\gamma_+$, $\Br(\omega_r)=\hat\gamma_-$. Then we have $H_\lambda=\{e\}$ and the following possibilities for $\Br(\lambda)$:
\begin{enumerate}
\item[1)]
If $m_{r-1}\equiv m_r\pmod{2}$, then
\[
\Br(\lambda)=\begin{cases}
\hat{\beta}^{\sum_{i=1}^{r/2} m_{2i-1}} & \text{if $r$ is even,}\\
\hat\beta^{\sum_{i=1}^{(r-1)/2} m_{2i-1}-(m_{r-1}-m_r)/2} & \text{if $r$ is odd.}
 \end{cases}
\]
\item[2)]
If $m_{r-1}\not\equiv m_r\pmod{2}$, then
\[
\Br(\lambda)=\begin{cases}
\hat{\beta}^{\sum_{i=1}^{r/2} m_{2i-1}}\hat\gamma_\pm & \text{if $r$ is even,}\\
\hat\gamma_\pm^{2\sum_{i=1}^{(r-1)/2} m_{2i-1}-m_{r-1}+m_r} & \text{if $r$ is odd,}
\end{cases}
\]
where we take $\hat\gamma_+$ for $\Gamma_+$ and $\hat\gamma_-$ for $\Gamma_-$.
\end{enumerate}
\end{theorem}

\begin{proof}
If $|T|>1$, assume we are dealing with $\Gamma_+$, the case of $\Gamma_-$ being completely analogous. Since $H_{\omega_i}$ is trivial for all $i$, we can apply Proposition \ref{prop:product}:
\[
\Br(\lambda)=\prod_{i=1}^r\Br(\omega_i)^{m_i}=\hat\beta^{\sum_{i=1}^{r-2}im_i}\hat\gamma_-^{m_{r-1}}\hat\gamma_+^{m_r}.
\]
The result follows from the fact $\hat\beta^2=1$ and the relations of Proposition \ref{prop:half_spin_relations}: $\hat\gamma_-=\hat\beta\hat\gamma_+$, $\hat\gamma_+^2=1$ if $r$ is even and $\hat\gamma_-=\hat\gamma_+^{-1}$, $\hat\gamma_+^2=\hat\beta$ if $r$ is odd.
\end{proof}

\begin{corollary}
The simple $\cL$-module $V_\lambda$ admits a $G$-grading making it a graded $\cL$-module if and only if one of the following conditions is satisfied:
\begin{enumerate}
\item[1)] $T=\{e\}$ and one of \textup{(i)} $m_{r-1}\equiv m_r\pmod{2}$ or \textup{(ii)} $\hat\gamma_+=1$;
\item[2)] $T\ne\{e\}$, $m_{r-1}\equiv m_r\pmod{2}$ and one of \textup{(i)} $r$ is even and $\sum_{i=1}^{r/2} m_{2i-1}$ is even or \textup{(ii)} $r$ is odd and $\sum_{i=1}^{(r-1)/2} m_{2i-1}-(m_{r-1}-m_r)/2$ is even;
\item[3)] $T\ne\{e\}$, $m_{r-1}\not\equiv m_r\pmod{2}$, $r$ is even, $\hat\gamma_+=1$ and one of \textup{(i)} we are dealing with $\Gamma^+$ and $\sum_{i=1}^{r/2} m_{2i-1}$ is even or \textup{(ii)} we are dealing with $\Gamma^-$ and $\sum_{i=1}^{r/2} m_{2i-1}$ is odd.
\end{enumerate}
\end{corollary}

\subsection{Outer gradings on $\So_{2r}(\FF)$}\label{ss:D_II}

Now suppose that the given grading on $\cL=\So_{2r}(\FF)$ is outer, i.e., the image of $\wh{G}$ in $\Aut(\cL)$ is not contained in $\inaut(\cL)$. We remind the reader that, in the case $r=4$, we assume that we have a matrix grading, i.e., the action of $\wh{G}$ fixes the natural module $V=V_{\omega_1}$. (This is automatic for $r\ne 4$.) So, the elements $\chi\in\wh{G}$ act on $\cL$ as conjugations by similitudes $u_\chi$ of $V$ and there exists $\chi\in\wh{G}$ such that $u_\chi$ is improper. Recall that this can happen only for $|T|=1$ or $4$ and, moreover, there is a distinguished element $h\in G$ of order $2$, which is characterized by the property that the corresponding $\bG$-grading is inner, where $\bG=G/\langle h \rangle$ (see Lemma \ref{lm:h} and Definition \ref{df:distinguished_element_D}). Also recall that $H_{\omega_i}=\{e\}$ for $i<r-1$ and $H_{\omega_{r-1}}=H_{\omega_r}=\langle h\rangle$. Denote $K=\langle h\rangle^\perp\subset\wh{G}$, so $K=K_{\omega_{r-1}}=K_{\omega_r}$.

Since $\{\omega_{r-1},\omega_r\}$ is a $\wh{G}$-orbit, we have $\Br(\omega_{r-1})=\Br(\omega_r)$ in the $\bG$-graded Brauer group. This implies that $\Br(\omega_1)$ is trivial in the $\bG$-graded Brauer group, i.e., the $\bG$-grading on $\cL$ is induced from a $\bG$-grading on $V$. (This can also be seen from the fact that $h\in T$ if $|T|=4$.)

\medskip

$\blacklozenge$ Consider the case $|T|=1$.

Already the $G$-grading on $\cL$ is induced by a $G$-grading on $V$, so it is determined by $g_0\in G$ and a multiset $\Xi=\{g_1,\ldots,g_{2r}\}$ satisfying \eqref{eq:relation_g0_trivT}. For the $\bG$-grading, we just have to take $g_0'=\bg_0\in\bG$ and $\Xi'=\{\bg_1,\ldots,\bg_{2r}\}$.

\medskip

$\blacklozenge$ Consider the case $|T|=4$.

Here $T=\{e,a,b,c\}$, $\beta(a)=\beta(b)=1$ (i.e., $X_a$ and $X_b$ are symmetric). The $G$-grading on $\cL$ is determined by $g_0\in G$ and a multiset $\Xi=\{g_1,\ldots,g_r\}$ satisfying \eqref{eq:relation_g0} with $k=r$. The operators $u_\psi$, $\psi\in K$, commute with each other, so they define a common eigenspace decomposition of $V$. Explicitly, writing $V=W\ot_\cD N$ and $u_\psi=\diag(\psi(g_1),\ldots,\psi(g_r))\ot X_t$, $t\in\{e,h\}$, we see that $v'_{2i-1}\bydef v_i\ot e_1$ and $v'_{2i}\bydef v_i\ot e_2$, $i=1,\ldots,r$, are common eigenvectors, where $e_1,e_2\in N$ are eigenvectors of $X_h$. Define a $\bG$-grading on $V$ by setting $\deg v'_{2i-1}=\bg_i$ and $\deg v'_{2i}=\bg_i\bar{h}'$, where $h'$ is an element of $T$ distinct from $e$ and $h$. (There are two such elements but they yield the same element of $\bG$.) One checks that the bilinear form $(\cdot,\cdot)_\Phi$ on $V$, with matrix \eqref{eq:PhiD}, is homogeneous with respect to this  $\bG$-grading, of degree $g'_0=\bg_0$ if $h\in\{a,b\}$ and $g'_0=\bg_0\bar{h}'$ if $h=c$. Hence the $\bG$-grading on $V$ induces a $\bG$-grading on $\cL$. Since
$\diag(\psi(g_1),\ldots,\psi(g_r))\ot\big(\begin{smallmatrix}1 & 0\\ 0 & \psi(h')\end{smallmatrix}\big)$
is a scalar multiple of $u_\psi$, for any $\psi\in K$, we see that this $\bG$-grading on $\cL$ coincides with the one induced from the original $G$-grading by the quotient map $G\to\bG$. The new multiset
\[
\Xi'=\{\bg_1,\bg_1\bar{h}',\ldots,\bg_r,\bg_r\bar{h}'\},
\]
after reordering, satisfies \eqref{eq:relation_g0_trivT} for $g'_0$ indicated above, possibly with a different $q$.

\medskip

In either case, we define $f_{\Xi',g'_0}\colon K\to\ZZ_2^{q'}/\langle\mathbf{1}\rangle$ similarly to equation \eqref{eq:homom_D} and set
\begin{equation}\label{eq:gammahat_0}
\hat\gamma_0(\psi_1,\psi_2)=(-1)^{f_{\Xi',g'_0}(\psi_1)\bullet f_{\Xi',g'_0}(\psi_2)}
\end{equation}
for all $\psi_1,\psi_2\in K$, similarly to equation \eqref{eq:gammahat_T1}.

\begin{theorem}\label{th:D_outer}
Let $\cL$ be the simple Lie algebra of type $D_r$ ($r\ge 3$) over an algebraically closed field $\FF$ of characteristic $0$. Suppose $\cL$ is graded by an abelian group $G$ such that the image of $\wh{G}$ in $\Aut(\cL)$ contains outer automorphisms. If $r=4$, assume further that this image is contained in the stabilizer of the natural module. Let $K=\langle h\rangle^\perp$, where $h\in G$ is the distinguished element. Then, for a dominant integral weight $\lambda=\sum_{i=1}^r m_i\omega_i$, we have the following possibilities:
\begin{enumerate}
\item[1a)]
If $m_{r-1}\ne m_{r}$ but $m_{r-1}\equiv m_r\pmod{2}$, then $H_\lambda=\langle h\rangle$, $K_\lambda=K$, and $\Br(\lambda)=1$.
\item[1b)]
If $m_{r-1}\not\equiv m_r\pmod{2}$, then $H_\lambda=\langle h\rangle$, $K_\lambda=K$, and $\Br(\lambda)=\hat\gamma_0$,
where $\hat\gamma_0\colon K\times K\to\FF^\times$ is given by equation \eqref{eq:gammahat_0}.
\item[2)]
If $m_{r-1}=m_{r}$, then $H_\lambda=\{e\}$ and $\Br(\lambda)=\hat\beta^{\sum_{i=1}^{\lfloor r/2 \rfloor}m_{2i-1}}$, with $\hat\beta\colon\wh{G}\times\wh{G}\to\FF^\times$ being the commutation factor associated to the parameters $(T,\beta)$ of the grading on $\cL$.
\end{enumerate}
\end{theorem}

\begin{proof}

1) Since we are dealing with an inner grading by $\bG=G/\langle h\rangle$, we apply Theorem \ref{th:D_inner} and take into account that, with respect to our $\bG$-grading, $\hat\beta=1$ and hence $\hat\gamma_+=\hat\gamma_-$ are given by equation \eqref{eq:gammahat_T1}, of which equation \eqref{eq:gammahat_0} is a restatement with the data pertaining to the $\bG$-grading.

2) Here $\lambda=\sum_{i=1}^{r-2}m_i\omega_i+m_{r-1}(\omega_{r-1}+\omega_{r})$. Since $H_{\omega_i}$ for $i<r-1$ and $H_{\omega_{r-1}+\omega_{r}}$ are trivial, we can apply Proposition \ref{prop:product}:
\[
\Br(\lambda)=\left(\prod_{i=1}^{r-2}\Br(\omega_i)^{m_i}\right)\Br(\omega_{r-1}+\omega_{r})^{m_{r-1}}=
\hat\beta^{\sum_{i=1}^{\lfloor (r-1)/2 \rfloor}m_{2i-1}}\Br(\omega_{r-1}+\omega_{r})^{m_{r-1}}.
\]
Thus, it suffices to show that
\[
\Br(\omega_{r-1}+\omega_{r})=
\begin{cases}
1 & \text{if $r$ is odd,}\\
\hat\beta & \text{if $r$ is even.}
\end{cases}
\]
In the $\bG$-graded Brauer group, we can write $\Br(\omega_{r-1}+\omega_r)=\Br(\omega_{r-1})\Br(\omega_{r})$ by Proposition \ref{prop:product}, so $\Br(\omega_{r-1}+\omega_r)=1$ by Proposition \ref{prop:half_spin_relations}. (Recall that the restriction of $\hat\beta$ to $K$ is trivial.) If $r$ is odd, then we actually have $\Br(\omega_{r-1}+\omega_r)=1$ in the $G$-graded Brauer group, thanks to Proposition \ref{prop:product_with_dual}, since $V_{\omega_{r-1}}$ and $V_{\omega_r}$ are dual to each other. So assume that $r$ is even.

Denote $\hat\gamma=\Br(\omega_{r-1}+\omega_{r})$. The fact that $\Br(\omega_{r-1}+\omega_r)$ is trivial in the $\bG$-graded Brauer group means that $\hat\gamma(\psi_1,\psi_2)=1$ for all $\psi_1,\psi_2\in K$. Fix $\chi\in\wh{G}\setminus K$ (i.e., $\chi(h)=-1$). It will be sufficient to show that $\hat\gamma(\chi,\psi)=\hat\beta(\chi,\psi)$ for all $\psi\in K$. We will again use the models $S^\pm$ for the modules $V_{\omega_{r-1}}$ and $V_{\omega_r}$. By our choice of $\chi$, we have $S^\pm\cong (S^\mp)^\chi$, and we are going to apply  Proposition \ref{prop:orbit_size_2}.

Let $S=S^+\oplus S^-$. Then $\Cl(V)$ can be identified with $\End(S)$ as a $\ZZ_2$-graded algebra. Since the similitude $u_\chi$ is improper, the corresponding element $s_\chi$ is odd, i.e., it swaps $S^+$ and $S^-$. Recall that, by definition of the $G$-grading on $\Cl\subo(V)$, we have $\chi*x=s_\chi x s_\chi^{-1}$ for all $x\in\Cl\subo(V)$, where $*$ denotes the $\wh{G}$-action associated to the $G$-grading. Since the imbedding $\pi\colon\cL\to\Cl\subo(V)$ respects the $G$-grading, $s_\chi\colon S\to S^\chi$ is an isomorphism of $\cL$-modules, so we may use $u'=s_\chi|_{S^+}$ and $u''=s_\chi|_{S^-}$ in Proposition \ref{prop:orbit_size_2}.

According to that proposition, $\hat\gamma$ is the commutation factor associated to the $G$-graded algebra $\cR=\End(S^+\ot S^-)=\End(S^+)\ot\End(S^-)$, where the $G$-grading is obtained by refining the $\bG$-grading using the conjugation by the operator $u=(u''\ot u')\circ\tau$ on $S^+\ot S^-$, with $\tau\colon S^+\ot S^-\to S^-\ot S^+$ being the flip. It is convenient to consider $\cR$ as imbedded into $\End(S\ot S)$ by virtue of the decomposition
$S\ot S=(S^+\ot S^+)\oplus (S^+\ot S^-)\oplus (S^-\ot S^+)\oplus (S^-\ot S^-).$
Recall that the $\bG$-grading on $\End(S^\pm)$ is associated with the $K$-action $\psi*x=s_\psi x s_\psi^{-1}$ for $x\in\End(S^\pm)$. It follows that the $\bG$-grading on $\cR$ is associated with the $K$-action $\psi*x=\tilde{u}_\psi x \tilde{u}_\psi^{-1}$ where $\tilde{u}_\psi$ is the restriction of $s_\psi\ot s_\psi$ to $S^+\ot S^-$. By the definition of the refinement in Proposition \ref{prop:orbit_size_2}, the $G$-grading on $\cR$ is associated to the extension of this $K$-action to a $\wh{G}$-action obtained by setting $\chi*x=uxu^{-1}$. By definition of commutation factor, $\hat\gamma(\chi,\psi)$ is the group commutator of $u$ and $\tilde{u}_\psi$. But $u$ is the restriction of $(s_\chi\ot s_\chi)\circ\tau$ to $S^+\ot S^-$, so it will be sufficient to compute the group commutator of $s_\chi\ot s_\chi$ and $s_\psi\ot s_\psi$ (since the flip clearly commutes with these operators). But by Lemma \ref{lm:commutator_s}, we know that $[s_\chi,s_\psi]\in\{\pm 1\}$ if $\hat\beta(\chi,\psi)=1$ and $[s_\chi,s_\psi]\in\{\pm z\}$ if $\hat\beta(\chi,\psi)=-1$. It follows that $[s_\chi\ot s_\chi,s_\psi\ot s_\psi]$ is $1$ in the first case and $z\ot z$ in the second case. But $z$ acts as $1$ on $S^+$ and $-1$ on $S^-$, so the restriction of $z\ot z$ to $S^+\ot S^-$ is $-1$.
Therefore, $\hat\gamma(\chi,\psi)=\hat\beta(\chi,\psi)$, as desired.
\end{proof}

\begin{corollary}
The simple $\cL$-module $V_\lambda$ admits a $G$-grading making it a graded $\cL$-module if and only if \textup{1)} $m_{r-1}=m_r$ and \textup{2)} $T=\{e\}$ or $\sum_{i=1}^{\lfloor r/2 \rfloor}m_{2i-1}$ is even.
\end{corollary}

\section*{Acknowledgments}

The second author would like to thank the Instituto Universitario de Matem\'aticas y Aplicaciones and Departamento de Matem\'{a}ticas of the University of Zaragoza for support and hospitality during his visit in January--April 2013.


\begin{thebibliography}{AB}

\bibitem{BK10} Bahturin, Y. and Kochetov, M. \emph{Classification of group gradings on simple {L}ie algebras of types {$\mathcal{A},\mathcal{B},\mathcal{C}$} and {$\mathcal{D}$}}. J. Algebra \textbf{324} (2010), no.~11, 2971--2989.

\bibitem{BL07} Billig, Y. and Lau, M. \emph{Thin coverings of modules}. J. Algebra \textbf{316} (2007), no.~1, 147--173.

\bibitem{Bou7_9} Bourbaki, N. \emph {{L}ie groups and {L}ie algebras}. Chapters 7--9.
Translated from the 1975 and 1982 French originals by Andrew Pressley. Elements of Mathematics (Berlin). Springer-Verlag, Berlin,  2005.

\bibitem{E09d} Elduque, A. \emph{Fine gradings on simple classical {L}ie algebras}. J. Algebra \textbf{324} (2010), no.~12, 3532--3571.

\bibitem{EKmon} Elduque, A. and Kochetov, M. \emph{Gradings on simple {L}ie algebras}. Mathematical Surveys and Monographs \textbf{189},  American Mathematical Society, Providence, RI, 2013.

\bibitem{KMRT} Knus, M.-A.;  Merkurjev, A.;  Rost, M.;  Tignol, J.-P. \emph{The book of involutions.} (With a preface in French by J. Tits.)
American Mathematical Society Colloquium Publications \textbf{44}, American Mathematical Society, Providence, RI, 1998.

\bibitem{Long74} Long, F. W.  \emph{A generalization of the Brauer group of graded algebras}. Proc. London Math. Soc. (3)  {\bf 29}  (1974), 237--256.

\bibitem{NVO82} N{\u{a}}st{\u{a}}sescu, C. and van Oystaeyen, F. \emph{Graded ring theory}. North-Holland Mathematical Library \textbf{28}, North-Holland Publishing Co., Amsterdam, 1982.

\bibitem{RR02} Ram, A. and Ramagge, J. \emph{Affine {H}ecke algebras, cyclotomic {H}ecke algebras and {C}lifford theory}. A tribute to C. S. Seshadri (Chennai, 2002), 428--466, Trends Math., Birkh\"auser, Basel,  2003.

\bibitem{Zol02} Zolotykh, A. A. \emph{Classification of projective representations of finite abelian groups}, Vestnik Moskov. Univ. Ser. I Mat. Mekh. (2002), no.~3, 3--10, 70.
\end{thebibliography}
\end{document}